\documentclass[11pt]{article}
\usepackage[english]{babel}
\usepackage{fullpage}
\usepackage{amsmath,amssymb, amsthm}
\usepackage{ulem}
\usepackage{cite}
\usepackage{graphicx}
\usepackage{authblk}
\usepackage{color}

\DeclareMathOperator{\Div}{div}	
\DeclareMathOperator{\dist}{dist}	
\DeclareMathOperator{\supp}{supp}	
\newcommand{\norm}[1]{\left\|#1\right\|}   
\newcommand{\abs}[1]{\left\lvert#1\right\rvert}  
\newcommand{\field}[1]{\mathbb{#1}}
\def\R{\field{R}} 
\def\Id{\field{I}} 
\newcommand{\vektor}[1]{{\mathbf{#1}}} 
\def\ro{\varrho}  
\def\en{\vektor{n}}  
\def\fib{\boldsymbol{\varphi}} 
\def\epsil{\varepsilon} 
\def\db{\mathtt{db}} 

\def\u{\vektor{u}} 
\def\uB{\vektor{u}_B} 
\def\uin{\vektor{u}_\infty} 
\newcommand{\refx}[1]{(\ref{#1})}
\def\vecx{\vektor{x}}
\def\vecX{\vektor{X}}
\def\vecy{\vektor{y}}
\newcommand{\de}[1]{\mathrm{d}{#1}}     
\def\dx{\de{\vecx}}	
\def\dt{\de{t}}
\def\ds{\de{s}}
\def\dtau{\de{\tau}}

\def\reg{r} 

\theoremstyle{plain}
\newtheorem{thm}{Theorem}
\newtheorem{lem}{Lemma}
\newtheorem{defin}{Definition}
\newtheorem{remark}{Remark}

\title{Rigid body in compressible flow with general inflow-outflow boundary data}

\author[1]{\v{S}imon Axmann}
\author[2]{\v{S}{\'a}rka Ne{\v{c}}asov{\'a}}
\author[3]{Ana Rado{\v{s}}evi{\'c}}

\date{}

 \affil[1]{University of Chemistry and Technology, Prague, Czechia}
 \affil[2]{Institute of Mathematics, Czech Academy of Sciences, Prague, Czechia}
 \affil[3]{Faculty of Economics and Business, University of Zagreb, Croatia}

\begin{document}
\maketitle

\section{Introduction} \label{intro}
We consider an initial-boundary value problem for a system of PDEs describing the motion of a rigid body placed in the compressible Newtonian fluid. We neglect the thermal effects, assuming the pressure given by the isentropic equation of state.  We deal with the motion of a rigid body in the compressible flow with the  general inflow-outflow boundary data. 
The novelty lies in the fact we are considering the  general inflow-outflow boundary data which were not considered before in the context of fluid-structure interaction.   

\subsection{Setting and notation}

We assume $\Omega\subset \field{R}^3$ to be a given bounded domain with sufficiently regular boundary. $S(t)\subset \Omega$ stands for the region which is occupied by the rigid body at time $t$, while $F(t)= \Omega\setminus S(t)$ for the part of the domain occupied by the fluid at time $t$.

Assuming sufficiently regular $S_0=S(0)$, we notice that the outer normal to $F(t)$ exists everywhere and we denote it by $\en$. Concerning the inhomogeneous boundary conditions, $\uB$ represents the prescribed velocity on the boundary and we denote
$$\Gamma_{in} = \{\vecx\in \partial\Omega \:|\: \uB \cdot \en <0 \}, \qquad \Gamma_{out} = \{\vecx\in \partial\Omega \:|\: \uB \cdot \en \geq 0 \}.$$
Furthermore, it is convenient to define the following time-space regions:
  \begin{align*}
Q_T =&\: (0,T)\times \Omega & Q^f =& \bigcup_{t\in(0,T)} \{t\}\times F(t) \\
 Q^S =& \bigcup_{t\in(0,T)} \{t\}\times S(t) & Q^{\partial S} =&\bigcup_{t\in(0,T)} \{t\}\times \partial S(t)
\end{align*}

In the classical formulation of the problem, one searches for 
 the density $\ro_f$ and the velocity field $\u_f$ of the fluid in $Q^f$, and the density  $\ro_s$ and the velocity field $\u_s$ of the rigid body in $Q^S.$ Naturally, the sets $S(t)$ and $F(t)$ need to be considered as unknown parts of the solution as well. 
 
 \medskip

\noindent For the rigid body we exploit the following usual notation:

$M=\int_{S(t)}{\ro_s(t,\vecx)}\dx\:$ denotes  the total mass of the rigid body,

$\vecX(t)=\int_{S(t)}{\vecx \ro_s(t,\vecx)}\dx\:$ denotes  the center of mass of the rigid body,

$\vektor{V}(t) = \dfrac{\de{}}{\de{t}} \vecX(t)\:$ denotes  the velocity of the center of mass of the rigid body,

$\field{J}(t) =\int_{S(t)}{ \ro_s(t,\vecx)\Bigl(| \vecx-\vecX(t)|^2 \,\Id  -  (\vecx-\vecX(t))^{T} (\vecx-\vecX(t))  \Bigr) }\dx \:$ denotes the inertia tensor of the rigid body, assuming $\vecx$, $\vecX$ to be row vectors.

\medskip

Furthermore, for the Newtonian fluid we denote
 by  $\field D(\u)=\frac{1}{2}\bigl(\nabla\u +\nabla^{T}\u \bigr)$ the symmetrical part of the velocity gradient, by $ \field S(\u_f) =  2\mu \field D(\u_f) + \lambda \Div\u_f\Id$ the viscous part of the stress tensor, and by
 $ p(\ro_f) = a \ro_f^{\gamma}$ the adiabatic pressure, we set $P(\ro_f) = a \dfrac{\ro_f^\gamma}{\gamma-1}$.

As the solid body is assumed to be rigid, we have
\begin{equation}
    \u_S(t,\vecx) = \vektor{V}(t)+ \field{Q}(t)\bigl(\vecx-\vecX(t)\bigr) \label{rigidveloc}
\end{equation}   with skew-symmetric tensor $\field{Q}$, whence there exists $\vektor{w}$ such that
$$\u_S(t,\vecx) = \vektor{V}(t)+\vektor{w}(t)\times (\vecx-\vecX(t)).$$

\medskip

\noindent We assume that initial position of the rigid body $S_0 \subset \Omega$ satisfies the following conditions:
\begin{itemize}
    \item There exists $r>0$ such that set $
O=\{ \vecx\in S_0| \dist (\vecx,\partial S_0 ) > r \}
$ is open, connected with $C^{2+\eta}$-boundary.
\item The distance of the boundary $\partial\Omega$ and solid $S_0$ is greater than $h>0.$
\end{itemize}
Our existence result holds as long as we are "away from collision", more precisely as long as the following condition is satisfied \begin{equation}\dist\bigl(S(t),\partial\Omega\bigr)\geq h.
\label{collision}\end{equation}
This assumption actually directly corresponds to the fact that we prescribe quite general outflow velocity field. In particular it is not necessarily rigid and thus not compatible with an "outflow" of the rigid body.

\medskip

\subsection{Brief overview of known results}

In this section we will mention some articles on the theory of existence of compressible Navier-Stokes equations and further comment on works devoted to fluid-structure interaction problems.\\
\begin{itemize}
	\item {\textit{Theory of  compressible Navier-Stokes equations:}} The global existence of {\bf strong solutions} for a small perturbation of a stable constant state was described in the celebrated work \cite{matnis}. In the article \cite{vallizak}, the authors established the local in time existence of strong solutions in the presence of inflow and outflow of the fluid through the boundary. Moreover,  in the same work they also gave the proof of global in time existence for small data in the absence of the inflow. 
 
 The theory of {\bf weak solutions} to the compressible Navier--Stokes equation equipped with homogeneous boundary data is starting from the pioneering works of Lions\cite{Lions93,MR1637634} 
  and Feireisl~\cite{FeNoPe01}.
 P.-L. Lions proved in \cite{Lions93,MR1637634} the global existence of renormalized weak solution with bounded energy for an isentropic fluid (i.e. $p(\rho)=\rho^{\gamma}$) with the adiabatic constant $\gamma>3d/(d+2),$ where $d$ is the space dimension.  After that by E. Feireisl $\mathit{et\, al.}$  the result was generalized to cover the range $\gamma>3/2$ in dimension $3$, and $\gamma>1$ in dimension $2$ in \cite{FeNoPe01}.

 Let us mention different approach presented in the recent work of Bresch, Jabin \cite{BreshJabin} where the authors introduce a completely new method to obtain compactness of the density. The well-posedness issues of the compressible Navier-Stokes equations for critical regularity data can be found in work of R. Danchin \cite{danchin}. For further references see  \cite{NoSt04, FeKaPo16}.
 
 The theory for general boundary data is much more recent. After 
initial results for some special cases \cite{Novo05, Giri11}, the existence of weak solutions for isentropic case  \cite{ChJiNo19, ChNoYa19}, as well as for the full Navier--Stokes--Fourier system \cite{FeNo21, ChFe22} was established. Let us also mention the general nonzero inflow-outflow problem, see \cite{KKNN}.

\item \textit{Rigid body in incompressible fluid:} The mathematical analysis of systems describing the motion of a rigid body in a viscous incompressible fluid is nowadays well developed.  The proof of existence of weak solutions until a first collision can be found in several papers, see \cite{CoJoTu99,DeEs99,GLSE,HOST,SER3}. Later, the possibility of collisions in the case of a weak solution was included, see \cite{F3,DEES2,SaStTu02}. 
For the case where the electromagnetic field is present in fluid, see \cite{BNSS}.

Let us also mention results on strong solutions, see e.g.\ \cite{GGH13,T, Wa}. 

\item \textit{Rigid body in compressible fluid:} Only a  few results are available on the motion of a rigid structure in a compressible fluid.  The existence of {\bf strong} solutions in the $L^2$-framework for small data up to a collision was shown in \cite{BG,roy2019stabilization}. The existence of strong solutions in the  $L^p$ setting based on $\mathcal{R}$-bounded operators was applied in the barotropic case \cite{ HiMu15} and   in the full system \cite{HaMaTaTu19}.

The existence of a {\bf weak} solution with the Dirichlet boundary conditions, also up to a collision but without smallness assumptions, was shown in \cite{DEES2}. Generalization of this result allowing collisions was given in \cite{Feir03, Sche19}.
The weak-strong uniqueness of a compressible fluid with a rigid body can be found in  \cite{ KrNePi20}. Existence of weak solutions in the case of the Navier boundary conditions  was established in \cite{NeRaMy22}. Moreover, for interaction of compressible heat-conducting fluid and rigid body, see
\cite{Brez08, HaMaTaTu19}.
\end{itemize}

\subsection{Problem formulation} 
The problem under consideration is described by the following equations:
\begin{align}
     \partial_t(\ro_f) + \Div(\ro_f \u_f) = 0, \quad  &\:\text{ in }Q^f \label{CEf}\\
\partial_t (\ro_f\u_f) +  \Div(\ro_f \u_f\otimes\u_f) - \Div \field S(\u_f)  +   \nabla  p(\ro_f)=  \vektor{0},\quad  &\:\text{ in }Q^f \\
\u_f = \u_S \quad  &\:\text{ on }Q^{\partial S} \\
M\dfrac{\de{}}{\dt} \vektor{V}(t) =   -\int_{\partial S(t)}{\Bigl(\field{S}(\u_f)  -  p(\ro_f)\Id \Bigr)\en}\de{S} &\:\text{ for }t\in(0,T) \\
\field{J}(t)\dfrac{\de{}}{\dt} \vektor{w}(t) =   \field{J}(t) \vektor{w}(t) \times \vektor{w}(t)-\int_{\partial S(t)}{\Bigl(\vecx-\vecX(t)\Bigr) \times\Bigl(\field{S}(\u_f)  -  p(\ro_f)\Id \Bigr)\en}\de{S} &\:\text{ for }t\in(0,T)
\end{align}
accompanied with initial conditions
\begin{gather}
  \ro_f(0)=\ro_0 \quad\text{ and }\quad\ro_f\u_f(0)=\vektor{m}_0 \text{ in }F(0),\qquad  \ro_s(0)=\ro_0 \text{ in }S(0)\\
   \vektor{V}(0)=\vektor{V_0},\qquad   \vektor{w}(0)=\vektor{w_0}, \qquad \Bigl(\frac{\vektor{m}_0(\vecx)}{\ro_0(\vecx)} = \vektor{V_0}  +  \vektor{w_0}\times(\vecx-\vecX(0)) \text{ for } \vecx\in\partial S(0)\Bigr)
\end{gather}
and boundary conditions
\begin{align}
\u_f = \uB\quad &\text{ on } (0,T)\times\partial\Omega\\
\ro_f = \ro_B \quad &\text{ on } (0,T)\times\Gamma_{in}\label{coupling}
\end{align}

For the corresponding weak formulation, we need to extend the boundary velocity in a suitable way inside the domain. This is provided by the following lemma.

\begin{lem}\label{extend}
Let $\Omega\subset\R^3$ be a bounded domain with Lipschitz boundary and $\uB\in W^{1,\infty}(\partial\Omega;\R^3)$.
Then there exists $h_0>0$ such that for any $h\in(0,h_0)$ there is an
extension vector field $\uin\in W^{1,\infty}(\R^3)\cap C_c(\R^3)$, such that
\begin{align*}
\uin|_{\partial \Omega}=&\:\uB,\\
\Div \uin \geq&\: 0 \text{ a.e. in } U_h = \{ \vecx\in\R^3| \dist (\vecx,\partial\Omega ) < h \},\\
\uin =&\: \boldsymbol{0} \text{ on }\R^3\setminus U_{2h}.
\end{align*}
\end{lem}
\begin{proof}
We can simply take the extension $\vektor{V}_\infty$ used in previous works (see \cite{Giri11})
and multiply it by a suitable smooth cut-off function $\xi\in C_c^{\infty}(\R^3)$, such that
$\xi=1$ on $U_h$ and $\xi=0$ on $\R^3\setminus U_{2h}$.
Then $\uin = \xi \vektor{V}_\infty$ clearly satisfies all of the conditions, including the last one.
\end{proof}

\begin{remark}
	The condition $\Div\uin\geq0$ is used to ensure the non-negativity of the term $p(\ro)\Div\uin$ near the boundary $\partial\Omega$. This condition plays a vital role in the passage to the limit of the integral involving $p(\varrho)\Div\uin$ in the energy inequality when transitioning from the approximating system to the original system (cf. Section \ref{Sec:ArtificialViscosity}).
	Namely, unlike the no-slip boundary conditions, the general boundary conditions do not allow for better pressure estimates than $L^{\infty}(0,T;L^1(\Omega))$ up to the boundary. 
	The improved estimate of pressure in the $L^q$ norm, which enables us to take the limit in the integral involving $p(\ro)\Div\uin$, is restricted to the interior of the domain, see \eqref{estimate_rho_1}.
\end{remark}

The outline of the article is the following: In Section \ref{intro} the model is described. The corresponding weak formulation together with the main result follows in Section \ref{Main result}. The proof of the main result is established in Section \ref{Construction}. First, we introduce  the approximation scheme (for parameters: $\delta, \epsilon, n, r $), including the definition of the approximation problem and we state the existence of the solution to the approximate problem. Further, we show the existence of the approximate problem. After that we pass to the limit with $n$ to identify the rigid body. Moreover, we pass to the limit with artificial viscosity ($\epsilon \to 0$) and as the last step we pass to the limit with the artificial pressure ($\delta \to 0$).

\section{Main result}\label{Main result}
 \subsection{Definition of weak solution}
 In this section we will introduce the definition of a weak solution to our problem. We will search for unknown density and velocity field in the whole domain defined in the following way
$$\ro = \begin{cases} \ro_f& \text{ in } Q^f\\
                      \ro_s& \text{ in } Q^S
\end{cases}
\qquad\text{ and }\qquad
\u = \begin{cases}  \u_f& \text{ in } Q^f\\
                    \u_s& \text{ in } Q^S
\end{cases}. $$

\begin{defin}\label{def_weak}
We say that triple $(\ro,\u,\eta)$ is a weak solution to problem \refx{CEf}-\refx{coupling} if $\uin$ is an~extension of the boundary data $\uB$ from Lemma \ref{extend} and
\begin{enumerate}
  \item functions $\ro$ and $\u$ belong to regularity classes
  \begin{gather*}
  \ro\geq0 \text{ a.e. in }Q_T,\quad \ro\in L^{\infty}\bigl({0,T;} L^{\gamma}(\Omega)\bigr), \\
  \u-\uin \in L^2\bigl({0,T;}W^{1,2}_0(\Omega)\bigr), \quad \ro\u \in C_{\text{weak}}({[0,T];}L^1(\Omega))
  \end{gather*}
and velocity field $\u$ is {\it compatible} with parametrized isometries $\eta[t]$ describing the evolution of $S(t)$, such that the mapping $t \mapsto \eta[t](y)
\text{ is absolutely continuous on } [0, T] $ and
$$
S(t) = \eta[t](S_0),\:\:\eta[t](\vecy) = \vecX(t) + \field{O}(t)\vecy ,\quad \field{O}\in SO(n),\: t\in[0,T],
$$
$$ \Bigl( \dfrac{\de{}}{\de{t}} \eta[t] \Bigr)\eta[t]^{-1}(\vecx)=\u(t,\vecx),\: \vecx\in \overline{S}
(t) \text{ and a.e. }t\in(0,T) ,$$
\item functions $\ro,\:\u$ satisfy in the sense of distributions the continuity equation
\begin{multline}\label{continuity_weak}
\int_{\Omega}{(\ro\psi)(\tau,.)}\dx - \int_{\Omega}{\ro_0(.)\psi(0,.)}\dx = \int_{0}^\tau \int_{\Omega} \bigl(\ro\partial_t \psi +\ro\u \cdot\nabla\psi  \bigr) \dx\dt - \int_0^\tau\int_{\Gamma_{in}}\ro_B\uB\cdot\en \psi \de{S}\dt\\
\forall\tau\in[0,T],\:\forall \psi \in C^1_{c}\bigl([0,T]\times(\Omega\cup\Gamma_{in})\bigr)
\end{multline}
as well as its renormalized version
\begin{multline} \label{renormalized_weak}
\int_{\Omega}{(b(\ro)\psi)(\tau,.)}\dx - \int_{\Omega}{b(\ro_0)(.)\psi(0,.)}\dx \\
 = \int_{0}^\tau \int_{\Omega} \Bigl(b(\ro)\partial_t \psi +b(\ro)\u \cdot\nabla\psi -  \psi\bigl(b'(\ro)\ro-b(\ro)\bigr)\Div\u  \Bigr) \dx\dt - \int_0^\tau\int_{\Gamma_{in}}b(\ro_B)\uB\cdot\en \psi \de{S}\dt\\
\forall\tau\in[0,T],\:\forall \psi \in C^1_{c}\bigl([0,T]\times(\Omega\cup\Gamma_{in})\bigr), \forall b\in C^1\bigl([0,\infty)\bigr) \text{ with } b'\in C^1_c\bigl([0,\infty)\bigr),
\end{multline}
 \item the weak variational formulation of the momentum equation holds true
\begin{multline}
\int_{\Omega}{(\ro\u\cdot\fib)(\tau,.)}\dx - \int_{\Omega}{\vektor{m}_0(.)\cdot\fib(0,.)}\dx \\
= \int_{0}^\tau \int_{\Omega} \Bigl(\ro\u \cdot\partial_t \fib +(\ro\u \otimes \u):\field{D}(\fib)  + p(\ro)\Div\fib - \field{S}(\u):\field{D}(\fib) \Bigr) \dx\dt \\
\forall\tau\in[0,T],\:\forall \fib \in \mathcal{R}(\overline{Q}^S) = \bigl\{\fib\in C^{\infty}_c(Q_T)|\field{D}(\fib)=0\text{ on some neighbourhood of } \overline{Q}^S \bigr\},
\end{multline}
\item we have the energy inequality
\begin{multline}
\int_\Omega \Bigl(\frac12\ro|\u-\uin|^2 + P(\ro) \Bigr)(\tau)\dx + \int_0^\tau\int_\Omega  \field{S}(\u - \uin): \field D (\u - \uin) \dx \dt \\
\leq \int_\Omega \Bigl(\frac12\ro_0|\u_0-\uin|^2 + P(\ro_0) \Bigr)\dx -\int_0^\tau\int_\Omega \ro\u\cdot\nabla\uin \cdot (\u-\uin)  \dx \dt\\
- \int_0^\tau \int_\Omega  p(\ro)\Div\uin \dx \dt - \int_0^\tau\int_\Omega  \field{S}(\uin): \field D (\u - \uin) \dx \dt\\
-\int_0^\tau\int_{\Gamma_{in}} P(\ro_B)\uB\cdot\en  \de{S} \dt,
\end{multline}
for a.e. $\tau\in(0,T).$
\end{enumerate}
\end{defin}

Our goal is to prove the following theorem.
\begin{thm} \label{main}
	Let $\Omega\subset\R^3$ be a bounded domain of class $C^{2+\nu}$, $\nu>0$ and let $h>0$. 
Suppose that
	$$  
	(\ro\u)_0 \in L^{2}(\Omega),\quad\ro_B\in C(\partial\Omega),
	\quad\uB\in C^2(\partial\Omega),
	$$ 
	$$
	\ro_0\geq 0,
	\quad \int_{\Omega}\ro_0\,\dx>0,
	\quad \ro_B\geq \underline{\ro}_B > 0,
	$$
	\begin{equation*}
		\int_{\Omega} \left(\frac{1}{2}\rho_0|\u_0|^2+P(\ro_0)\right)\,\dx <\infty
	\end{equation*}
	and the pressure $p$ is given by
	\begin{equation*}
	p(\ro) = a\ro^{\gamma},\quad\gamma>\frac{3}{2}.
	\end{equation*}
	Then for $\uin$ an~extension of the boundary data $\uB$ from Lemma \ref{extend} there exists a time $T>0$ and a weak solution in the sense of Definition \ref{def_weak} such that $\dist(S(t),\partial\Omega)\geq h$, for all $t\in[0,T]$.
\end{thm}

\section{Construction}\label{Construction}

\subsection{Approximation scheme}
For our approximation scheme, we introduce the following parameters:
\begin{itemize}
\item $\delta>0$ (pressure integrability) with $\beta>\max\bigl\{ \frac{9}{2}, \gamma\bigr\}$
\item $\epsil>0$ (artificial viscosity)
\item $n>0$ (artificial high viscosity for solidification) and $\reg>0$ (velocity regularization)
\end{itemize}
We denote $p_\delta(\ro) = p(\ro) + \delta \ro^\beta, $ and $P_\delta(\ro) = P(\ro)  + \delta  \dfrac{\ro^\beta}{\beta-1}.$
Furthermore, for a function $\chi$ representing the signed distance from the interface between the fluid and the rigid body, we will denote the artificial viscosities defined on the entire domain $\Omega$ as $\mu_n(\chi)$ and $\lambda_n(\chi)$. We will define them later in such a way that within the fluid domain, the viscosities match the constant fluid viscosities, $\mu$ and $\lambda$, while within the rigid body domain, they tend to infinity as $n$ tends to infinity. Moreover, we denote $\field{S}_{\chi,n}(\u) = 2\mu_n(\chi)\field{D}(\u)+\lambda_n(\chi)\Div\u\Id$.

Referring to \cite{ChJiNo19} and \cite{Feir03}, the starting point of this scheme is then a problem with all the approximation parameters:
\begin{align}
\partial_t\ro - \epsil\Delta\ro  + \Div(\ro\u)=&\:0 \text{ in }Q_T\label{CEapprox}\\
\ro(0) = &\:\ro_{0,\delta} \text{ in }\Omega 
\end{align}
\begin{equation}
(-\epsil \nabla\ro + \ro\u)\cdot \en = \begin{cases} \ro_B\uB\cdot\en & \text{ on } (0,T)\times \Gamma_{in}\\
\ro\uB\cdot\en &\text{ on }(0,T)\times \Gamma_{out}
\end{cases}
\label{BCCEapprox}
\end{equation}
\begin{equation}
\partial_t(\ro\u) + \Div(\ro\u\otimes\u) + \nabla p_\delta(\ro)  = \Div \field{S}_{\chi,n}(\u)
-\epsil \nabla\ro\cdot\nabla\u  \text{ in }Q_T 
\end{equation}
\begin{equation}
 \u(0) = \u_0 \text{ in }\Omega , \qquad \u = \uB \text{ on }(0,T)\times\partial \Omega, \label{BCapprox}\end{equation}
 where $\ro_{0,\delta}\in C^{1}\bigl(\overline{\Omega}\bigr)$ is a suitable regularization of the initial datum $\ro_0$ such that
 $$0< \delta\leq\ro_{0,\delta}(\vecx)\leq\frac{1}{\delta},\qquad  \forall \vecx \in\Omega.$$
In contrast to \cite{ChJiNo19}, we do not introduce regularizing term $\epsil\Div\bigl( |\nabla(\u-\uin)|^2\nabla (\u-\uin)\bigr)$  to the momentum equation. However, inspired by \cite{FeNo22} in order to guarantee the required regularity of boundary data, we approximate condition \refx{BCCEapprox} further on the Galerkin level. Note that we do not need any compatibility condition involving given initial and boundary velocity, as initial velocity $\u_0\in L^2(\Omega)$ has no trace.

\subsubsection*{General strategy}
Our general strategy to prove Theorem \ref{main} will follow the steps below:
\begin{enumerate}
\item The existence for fixed $\delta,\epsil,\reg, n$ 
is established via Galerkin method in the spirit of Section 4 of \cite{ChJiNo19}, the main difference in the result lies in the fact that we need to treat the non-constant viscosities.
\item  We identify the rigid body by solidification of the appropriate part of the fluid. More precisely, we pass with $n\to\infty$ and simultaneously achieve $\chi$ to be connected to the velocity regularisation $R_r[\u]$ via version of Proposition 5.1 from \cite{Feir03}.  
\item We pass with $\epsil\to 0$ to get rid of the artificial mass diffusion.
\item We pass with $\delta \to 0$ to obtain a solution to the original problem.
\end{enumerate}

\subsubsection*{Definition of weak solutions to the approximate problem}
We define weak solutions to the approximate problem as follows.
\begin{defin} \label{defapprox}
We say that pair $(\ro, \u)$ is a solution to the approximate problem \refx{CEapprox}-\refx{BCapprox} if  $\uin$ is an~extension of the boundary data $\uB$ from Lemma \ref{extend} and
\begin{enumerate}
\item    
    $\ro\geq0 \text{ a.e. in }Q_T,\quad \ro\in L^{\infty}\bigl(0,T;L^{\beta}(\Omega)\bigr)  \cap  L^2\bigl(0,T;W^{1,2}(\Omega)\bigr),\quad   \ro \in C_{\text{weak}}([0,T];L^\beta(\Omega))$,\\
    $\u\in 
    L^2\bigl(0,T;W^{1,2}(\Omega)\bigr), \quad \ro\u \in C_{\text{weak}}([0,T];L^1(\Omega)),    $
\item functions $(\ro,\u)$ satisfy the continuity equation
\begin{multline}
\int_{\Omega}{(\ro\psi)(\tau,.)}\dx - \int_{\Omega}{\ro_{0,\delta}(.)\psi(0,.)}\dx + \epsil \int_0^\tau \int_{\Omega}\nabla\ro\cdot \nabla\psi \dx\dt +  \int_0^\tau\int_{\Gamma_{in}}\ro_B\uB\cdot\en \psi \de{S}\dt \\= \int_{0}^\tau \int_{\Omega} \bigl(\ro\partial_t \psi +\ro\u \cdot\nabla\psi  \bigr) \dx\dt ,\quad
\forall\tau\in[0,T],\:\forall \psi \in C^1_{c}\bigl([0,T]\times(\Omega\cup\Gamma_{in})\bigr)\label{CEeps}
\end{multline}
as well as its renormalized version
\begin{multline}
\int_{\Omega}{(b(\ro)\psi)(\tau,.)}\dx - \int_{\Omega}{b(\ro_{0,\delta})(.)\psi(0,.)}\dx + \int_0^\tau\int_{\partial\Omega}\Bigl(b(\ro)\uB-\epsil b'(\ro)\nabla \ro\Bigr)\cdot\en \psi \de{S}\dt \\
 = \int_{0}^\tau \int_{\Omega} \Bigl(b(\ro)\partial_t \psi +\bigl(b(\ro) \u -\epsil b'(\ro) \nabla \ro \bigr) \cdot\nabla\psi -  \psi\bigl(b'(\ro)\ro-b(\ro)\bigr)\Div\u   - \psi \epsil b''(\ro)|\nabla\ro|^2 \Bigr) \dx\dt \\
\forall\tau\in[0,T],\:\forall \psi \in C^1\bigl([0,T]\times\overline{\Omega}\bigr), \forall b\in C^2\bigl([0,\infty)\bigr) \text{ with } b'\in C^1_c\bigl([0,\infty)\bigr),\label{RCEeps}
\end{multline}
\item the weak variational formulation of the momentum equation holds true
\begin{multline}
\int_{\Omega}{(\ro\u\cdot\fib)(\tau,.)}\dx - \int_{\Omega}{\vektor{m}_0(.)\cdot\fib(0,.)}\dx
=  - \epsil \int_{0}^\tau \int_{\Omega} \bigl(  \nabla \ro\cdot \nabla \u \cdot \fib\bigr) \dx\dt\\ + \int_{0}^\tau \int_{\Omega} \Bigl(\ro\u \cdot\partial_t \fib +(\ro\u\otimes\u) :\field{D}(\fib)  + p_\delta(\ro)\Div\fib -\field{S}_{\chi,n}(\u):\field{D}(\fib) \Bigr) \dx\dt \\
\forall\tau\in[0,T],\:\forall \fib \in  C^{\infty}_c(Q_T),
\end{multline}
\item the energy inequality 
reads
\begin{equation}
\begin{split}
&\int_\Omega \Bigl(\frac12\ro|\u-\uin|^2 + P_\delta(\ro) \Bigr)(\tau)\dx  + \epsil  \int_0^\tau\int_\Omega  P''_\delta(\ro)|\nabla\ro|^2 \dx \dt\\
&\qquad+ \int_0^\tau\int_\Omega \bigl(\field{S}_{\chi,n}(\u-\uin) \bigr): \field D (\u - \uin)\dx \dt +   \int_0^\tau\int_{\Gamma_{out}} P_\delta(\ro)\uB\cdot\en  \,\de{S} \dt  \\
&\qquad+\int_0^\tau\int_{\Gamma_{in}}\bigl( P_\delta(\ro_B) - P'_\delta(\ro) (\ro_B -\ro) - P_\delta(\ro)\bigr)|\uB\cdot\en | \,\de{S} \dt\\
&\qquad\qquad\leq \int_\Omega \Bigl(\frac12\ro_{0,\delta}|\u_0-\uin|^2 + P_\delta(\ro_{0,\delta}) \Bigr)\dx -\int_0^\tau\int_\Omega \ro\u\cdot\nabla\uin \cdot (\u-\uin)  \dx \dt\\
&\qquad\qquad\qquad- \int_0^\tau \int_\Omega  p_\delta(\ro)\Div\uin \dx \dt -\int_0^\tau\int_{\Gamma_{in}} P_\delta(\ro_B)\uB\cdot\en  \,\de{S} \dt\\
&\qquad\qquad\qquad - \int_0^\tau\int_\Omega  \field{S}_{\chi,n}(\uin): \field D (\u - \uin) \dx \dt+ \epsil \int_0^\tau \int_\Omega \nabla \ro \cdot\nabla (\u-\uin)\cdot\uin \,\dx\dt.
\label{EIeps}
\end{split}
\end{equation}
\end{enumerate}
\end{defin}

\begin{remark}
The energy inequality \eqref{EIeps} is obtained formally by taking $\fib = \u-\uin$ in the momentum equation, $\psi =\dfrac{|\uin|^2-|\u|^2}{2}$ in the continuity equation and $b(\ro) =  P_\delta(\ro)$ and $\psi = 1$ in the renormalized continuity equation.
\end{remark}

Let us note that in fact on $\Gamma_{in}$, we have $|\uB\cdot \en| = -\uB\cdot \en.$ 
Furthermore, due to the convexity of $P_\delta$,
we have $\bigl( P_\delta(\ro_B) - P'_\delta(\ro) (\ro_B -\ro) - P_\delta(\ro)\bigr)\geq0$, thus we can infer
\begin{equation}
\begin{split}\label{ei1}
&\int_\Omega \Bigl(\frac12\ro|\u-\uin|^2 + P_\delta(\ro) \Bigr)(\tau)\dx  + \epsil  \int_0^\tau\int_\Omega   P''_\delta(\ro)|\nabla\ro|^2 \dx \dt\\
&\qquad+ \int_0^\tau\int_\Omega \bigl(  2\mu_n(\chi)|\field{D}(\u-\uin)|^2+\lambda_n(\chi)\bigl(\Div(\u-\uin)\bigr)^2\dx \dt +   \int_0^\tau\int_{\Gamma_{out}} P_\delta(\ro)\uB\cdot\en  \,\de{S} \dt  \\
&\qquad\qquad\leq \int_\Omega \Bigl(\frac12\ro_{0,\delta}|\u_0-\uin|^2 + P_\delta(\ro_{0,\delta}) \Bigr)\dx -\int_0^\tau\int_\Omega \ro\u\cdot\nabla\uin \cdot (\u-\uin)  \dx \dt\\
&\qquad\qquad\qquad- \int_0^\tau \int_\Omega  p_\delta(\ro)\Div\uin \dx \dt -\int_0^\tau\int_{\Gamma_{in}} P_\delta(\ro_B)\uB\cdot\en  \,\de{S} \dt\\
&\qquad\qquad\qquad - \int_0^\tau\int_\Omega  \field{S}_{\chi,n}(\uin): \field D (\u - \uin) \dx \dt+ \epsil \int_0^\tau \int_\Omega \nabla \ro \cdot\nabla (\u-\uin)\cdot\uin \,\dx\dt.
 \end{split}
\end{equation}

\subsection{Existence of the solution to the approximate problem}

The existence of the solution to the approximate problem defined in the previous section can be derived in a similar way as in \cite{ChJiNo19}. The main differences are the presence of non-constant viscosities, the lack of regularization of the velocity gradient, and a different approach to the boundary condition for the continuity equation. 

As for the latter, we introduce as in \cite{FeNo22} a sequence of smooth functions $[\,\cdot\,]_N^{-}\in C^{\infty}(\R)$ approximating the function of negative part, namely
$$[v]_N^{-} \leq \min\{v, 0\}, \quad [v]_N^{-}=\begin{cases} v &\text{ for } v\leq - \frac{1}{N}\\
\text{increasing}&\text{ for }v\in\bigl[ -\frac{1}{N}, \frac{1}{N} \bigr]\\
0 &\text{ for } v\geq\frac{1}{N}
\end{cases}
$$
where $N$ denotes the dimension of the finite dimensional space for the Galerkin approximation. The continuity equation is then accompanied with regularized boundary condition
$$\epsil\nabla\ro\cdot\en = (\ro-\ro_B) [\uB\cdot \en]_N^{-}\text{ on }\partial\Omega$$
and regularized initial condition $\ro_{0,\delta,N}$ satisfying also the compatibility condition
$$\epsil\nabla\ro_{0,\delta,N}\cdot\en = (\ro_{0,\delta,N}-\ro_B) [\uB\cdot \en]_N^{-}\text{ on }\partial\Omega,$$
see \cite[Section 5.1.5]{FeNo22} for a precise construction of such $\ro_{0,\delta, N}.$ This enables the classical parabolic existence theory to be applied.

As in \cite{Feir03}, we allow the viscosities to depend smoothly on $\chi\in\R$	
$$\mu_n = \mu_n\bigl(\chi\bigr),\:\lambda_n = \lambda_n\bigl(\chi\bigr),$$	
where $\chi$ serves as an indicator that allows us to determine whether we are within the approximate fluid region, where the viscosity coefficients remain constant, or within the approximate solid region, where these coefficients are defined to grow arbitrarily large.
Then, we can proceed with the proof using the methodology employed in \cite[Section 4.3]{ChJiNo19}.  Importantly, the introduction of the additional variable $\chi$ does not require any substantial modifications to the approach.
Therefore, we only restate the result and add a few comments on the proof.

\begin{lem} \label{existapprox}
Let $\Omega\subset \R ^3$. Suppose that there exist positive numbers $0<\underline{\ro}<\overline{\ro}$ such that the initial and boundary data satisfy
$$\u_0 \in L^2(\Omega),\quad   \ro_{0,\delta} \in W^{1,2}(\Omega),\quad\ro_B\in C(\partial\Omega),\quad \ro_{0,\delta},\ro_B\in\Bigl[\underline{\ro},\overline{\ro}\Bigr] ,  $$ then for $\uin$ an~extension of the boundary data $\uB$ from Lemma \ref{extend} there exists $\epsil_0>0$ and a weak solution to the approximating problem in the sense of Definition \ref{defapprox}\ref{defapprox}, for all $\epsil\in(0,\epsil_0)$, $\delta\in(0,1)$. Moreover, assuming
\begin{align}
&\mu_n(\chi)\geq\mu_0>0,\quad
\mu_n(\chi)+\lambda_n(\chi)\geq 0
\\
&\mu_n(\chi)=\mu_0,\, \lambda_n(\chi)=\lambda_0 \quad\text{on }\supp{\uin}
\label{collision1}
\end{align}
for some constants $\mu_0>0$, $\lambda_0>0$, the solution satisfies the following estimates:
\begin{align}
	&\norm{\ro\abs{\u-\uin}^2}_{L^{\infty}(0,T;L^1(\Omega))}
	\leq c(data,\mu_0,T),
	\\
	&\sqrt[\beta]{\delta}\norm{\ro}_{L^{\infty}(0,T;L^{\beta}(\Omega))}
	\leq c(data,\mu_0,T),
	\\
	&\norm{\ro}_{L^{\infty}(0,T;L^{2}(\Omega))}
	\leq c(data,\delta,\mu_0,T),
	\\
	&\sqrt{\epsil\delta}\norm{\nabla\bigl(\ro^{\frac{\beta}{2}}\bigr)}_{L^{2}(0,T;L^{2}(\Omega))}
	\leq c(data,\mu_0,T),
	\\
	&\sqrt{\epsil}\norm{\nabla\ro}_{L^{2}(0,T;L^{2}(\Omega))}
	\leq c(data,\delta,\mu_0,T),
	\\
	&\norm{\u-\uin}_{L^{2}(0,T;W^{1,2}(\Omega))}
	\leq c(data,\mu_0,T),
	\\
	&\sqrt[\beta]{\delta}\norm{\ro\abs{\uB\cdot\en}^{\frac{1}{\beta}}}_{L^{\beta}(0,T;L^{\beta}(\partial\Omega))}
	\leq c(data,\mu_0,T),
	\\
	&\norm{\ro\abs{\uB\cdot\en}^{\frac{1}{2}}}_{L^{2}(0,T;L^{2}(\partial\Omega))}
	\leq c(data,\delta,\mu_0,T),
	\end{align}
	and since $\uin\in W^{1,\infty}(\Omega)$
	\begin{align}
	&\norm{\ro\abs{\u}^2}_{L^{\infty}(0,T;L^1(\Omega))}
	\leq c(data,\mu_0,T).
\end{align}

\end{lem}

\subsubsection{Estimates independent of $\mu_n(\chi)$ and $\lambda_n(\chi)$}
The proof of Lemma \ref{existapprox} is essentially the same as in \cite[Section 4.3]{ChJiNo19}, with the only distinction lying in the variable viscosities. Therefore, here we will only present the derivation of the estimates independent of the choice of the non-constant viscosities. 
In doing so, we will use the continuity equation and the energy inequality, which, at the level of Galerkin approximations, have the same form as at the limit, i.e., in the statement of Lemma \ref{existapprox}.

Let us derive the estimates independent of the choice of the non-constant viscosities as given in Lemma~\ref{existapprox}.

We deduce from the Korn and the Poincar\'e inequalities
\begin{multline}
\int_\Omega \bigl(  2\mu_n(\chi)|\field{D}(\u-\uin)|^2+\lambda_n(\chi)\bigl(\Div(\u-\uin)\bigr)^2\dx
\\
=
\int_\Omega \Bigl(
\underbrace{\mu_n(\chi)}_{\geq\mu_0}|\field{D}(\u-\uin)|^2
+\underbrace{(\mu_n(\chi) + \lambda_n(\chi))}_{\geq 0}\left(\Div(\u-\uin)\right)^2
\\
+ \underbrace{\mu_n(\chi)}_{>0}\underbrace{\left(|\field{D}(\u-\uin)|^2 -\left(\Div(\u-\uin)\right)^2\right)}_{\geq 0}
\Bigr)\dx
\\
\geq c\mu_0 \|\u-\uin\|_{W^{1,2}(\Omega)}^2.
\end{multline}

For the right-hand side of \eqref{ei1} we obtain
\begin{align}
\left|
\int_{0}^{\tau} \int_{\Omega} p_{\delta}\left(\ro\right) \operatorname{div} \mathbf{u}_{\infty} \dx \dt
\right|
&\leq c(data)\int_{0}^{\tau} \int_{\Omega}
P_{\delta}(\ro) \dx \dt
\\
-\int_0^\tau\int_{\Gamma_{in}} P_\delta(\ro_B)\uB\cdot\en  \,\de{S} \dt\,\,
&
\leq c(data,T), \qquad \text{for } \delta<1
\\
\left|
\int_{0}^{\tau} \int_{\Omega} \ro \u \cdot \nabla \uin \cdot(\u-\uin) \dx \dt
\right|
&\leq \int_{0}^{\tau} \int_{\Omega} \ro |\nabla \uin| |\u-\uin|^2 \dx \dt
\\\nonumber
&\qquad+ \int_{0}^{\tau} \int_{\Omega} \ro |\uin| |\nabla \uin ||\u-\uin| \dx \dt
\\\nonumber
&\leq c(data) \int_{0}^{\tau} \int_{\Omega} 
	\left(
	\ro|\u-\uin|^2 + \ro
	\right)
	\dx \dt
\\\nonumber
&\leq c(data) \int_{0}^{\tau} \int_{\Omega} 
	\left(
	\ro|\u-\uin|^2 + P_{\delta}(\ro)
	\right)
	\dx \dt
	+ c(data, T)
\\
\left|
\epsil \int_{0}^{\tau} \int_{\Omega} \nabla \ro \cdot \nabla(\u-\uin) \cdot \uin \dx \dt
\right|
&\leq
\alpha\epsil\|\nabla\ro\|_{L^2((0,\tau)\times\Omega)}^2
+ \epsil\frac{c(data)}{\alpha}
\|\nabla(\u-\uin)\|_{L^2((0,\tau)\times\Omega)}^2
\\
\left|
\int_{0}^{\tau} \int_{\Omega} \field S\left( \uin\right): \nabla(\u-\uin) \dx\dt
\right|
&\leq \int_{0}^{\tau} \int_{\Omega} |\field S\left(\uin\right)||\nabla(\u-\uin)| \dx\dt
\\\nonumber
&\leq \alpha \|\nabla(\u-\uin)\|_{L^2((0,\tau)\times\Omega)}^2
+ \frac{c(data,T)}{\alpha}
\end{align}
for arbitrary $\alpha>0$, since 
\begin{align*}
\int_{\Omega}|\field S\left( \uin\right)|^2\dx
&= \int_{\Omega}|2\mu_n(\chi)\field D\left(\nabla \uin\right) + \lambda_n(\chi)\Div\uin\,\Id|^2\dx
\\
&= \int_{\Omega}|2\mu_0\field D\left(\nabla \uin\right) + \lambda_0\Div\uin\,\Id|^2\dx
\leq c(data)
\end{align*}
by condition \refx{collision1}.
Thus, the right-hand side of \eqref{ei1} is bounded by
\begin{multline}
c(data) \int_{0}^{\tau} \int_{\Omega} 
\left(
\frac{1}{2}\ro|\u-\uin|^2 + P_{\delta}(\ro)
\right)\,\dx\dt
+\alpha\epsil\|\nabla\ro\|_{L^2((0,\tau)\times\Omega)}^2
\\
+ \left(
\alpha 
+ \epsil\frac{c(data)}{\alpha}
\right)
\|\nabla(\u-\uin)\|_{L^2((0,\tau)\times\Omega)}^2
+ \frac{c(data, T)}{\alpha} + c(data, T).
\end{multline}
All together, gives us
\begin{multline}
\int_\Omega \Bigl(\frac12\ro|\u-\uin|^2 + P_\delta(\ro) \Bigr)(\tau)\dx  
+ \epsil  \int_0^\tau\int_\Omega  P''_\delta(\ro)|\nabla\ro|^2\dx \dt\\
+ c\mu_0 \|\u-\uin\|_{L^2(0,\tau;W^{1,2}(\Omega))}^2
+ \int_0^\tau\int_{\Gamma_{out}} P_\delta(\ro)\uB\cdot\en  \,\de{S} \dt 
\\
\leq c(data) \int_{0}^{\tau} \int_{\Omega} 
\left(
\frac{1}{2}\ro|\u-\uin|^2 + P_{\delta}(\ro)
\right)\,\dx\dt
+\alpha\epsil\|\nabla\ro\|_{L^2((0,\tau)\times\Omega)}^2
\\
+ \left(
\alpha 
+ \epsil\frac{c(data)}{\alpha}
\right)
\|\nabla(\u-\uin)\|_{L^2((0,\tau)\times\Omega)}^2
+ \frac{c(data, T)}{\alpha} + c(data, T).
\end{multline}
Now, we take first $\alpha>0$ sufficiently small and  then $\varepsilon>0$ sufficiently small
\begin{multline}
\int_\Omega \Bigl(\frac12\ro|\u-\uin|^2 + P_\delta(\ro) \Bigr)(\tau)\dx  
+ \epsil  \int_0^\tau\int_\Omega   P''_\delta(\ro)|\nabla\ro|^2\dx \dt\\
+ c(data,\mu_0,T) \|\u-\uin\|_{L^2(0,\tau;W^{1,2}(\Omega))}^2 +   \int_0^\tau\int_{\Gamma_{out}} P_\delta(\ro)\uB\cdot\en  \,\de{S} \dt 
\\
\leq c(data) \int_{0}^{\tau} \int_{\Omega} 
\left(
\frac{1}{2}\ro|\u-\uin|^2 + P_{\delta}(\ro)
\right)\dx\dt
+ c(data,\mu_0, T).
\end{multline}
By the Gronwall inequality we get
\begin{multline}
\int_\Omega \Bigl(\frac12\ro|\u-\uin|^2 + P_\delta(\ro) \Bigr)(\tau)\dx  
+ \epsil  \int_0^\tau\int_\Omega   P''_\delta(\ro)|\nabla\ro|^2\\
+ c(data,\mu_0,T) \|\u-\uin\|_{L^2(0,\tau;W^{1,2}(\Omega))}^2 +   \int_0^\tau\int_{\Gamma_{out}} P_\delta(\ro)\uB\cdot\en  \,\de{S} \dt 
\leq c(data,\mu_0,T).
\end{multline}
Thus, we obtain the following estimate
\begin{multline}
\norm{\ro\abs{\u-\uin}^2}_{L^{\infty}(0,T;L^1(\Omega))} 
+ \norm{\u-\uin}_{L^{2}(0,T;W^{1,2}(\Omega))}
\\ 
+\sqrt[\beta]{\delta}\norm{\ro}_{L^{\infty}(0,T;L^{\beta}(\Omega))}
+   \sqrt{\epsil\delta}\norm{\nabla\bigl(\ro^{\frac{\beta}{2}}\bigr)}_{L^{2}(0,T;L^{2}(\Omega))} 
\\
+    \sqrt[\beta]{\delta}\norm{\ro\abs{\uB\cdot\en}^{\frac{1}{\beta}}}_{L^{\beta}(0,T;L^{\beta}(\partial\Omega))}
\leq  c(data,\mu_0,T)
\label{esteps}
\end{multline}
and since $\uin\in W^{1,\infty}(\Omega)$
\begin{equation}
\norm{\ro\abs{\u}^2}_{L^{\infty}(0,T;L^1(\Omega))}
+\norm{\u}_{L^{2}(0,T;W^{1,2}(\Omega))}
\leq c(data,\mu_0,T).\label{esteps2}
\end{equation}
The strong convergence of the gradient of density can be deduced following the same lines as in \cite{FeNo22}.

From the continuity equation (by taking $\psi=\ro$) we obtain for any $\tau\in[0,T]$
	\begin{multline}
	\frac12\int_{\Omega}\ro^2(\tau)\,\dx
	+\frac12\int_{0}^{\tau}\int_{\partial\Omega}\ro^2\abs{\uB\cdot\en}\,\de{S}\dt
	+\epsil \int_{0}^{\tau}\int_{\Omega} \abs{\nabla\ro}^2\,\dx\dt
	\\
	=\int_{\Omega}\frac12\ro_{0,\delta}^2\,\dx
	+ \int_0^\tau\int_{\Gamma_{in}} \ro\ro_B\abs{\uB\cdot\en}\,\de{S}\dt
	-\frac12\int_{0}^{\tau}\int_{\Omega}\ro^2\Div\u\,\dx\dt.
	\end{multline}
We estimate the right-hand side
	\begin{align*}
	\int_0^\tau\int_{\Gamma_{in}} \ro\ro_B\abs{\uB\cdot\en}\,\de{S}\dt
	& \leq \int_0^\tau\int_{\Gamma_{in}}
	\alpha\ro^2\abs{\uB\cdot\en}+
	\frac{c}{\alpha}\ro_B^2\abs{\uB\cdot\en}
	\,\de{S}\dt
	\nonumber\\
	& \leq \alpha\int_0^\tau\int_{\partial\Omega}
	\ro^2\abs{\uB\cdot\en}\,\de{S}\dt +
	\frac{c(data)}{\alpha},
	\end{align*}
	\begin{align*}
	\abs{\frac12\int_{0}^{\tau}\int_{\Omega}\ro^2\Div\u\,\dx\dt}
	&\leq c(data) \int_{0}^{\tau} \int_{\Omega}
	\frac12\ro^2
	\dx \dt
	+ \int_{0}^{\tau} \int_{\Omega}
	\ro^2\Div(\u-\uin)
	\dx \dt
	\nonumber\\
	&\leq c(data) \int_{0}^{\tau} \int_{\Omega}
	\frac12\ro^2
	\dx \dt
	\nonumber\\
	&\qquad+ c(data)\int_{0}^{\tau} \int_{\Omega}
	\ro^4
	\dx \dt
	+\alpha \|\nabla(\u-\uin)\|_{L^2((0,\tau)\times\Omega)}^2
	\nonumber
	\\
	&\leq c(data) \int_{0}^{\tau} \int_{\Omega}
	\frac12\ro^2
	\dx \dt
	\nonumber\\
	&\qquad
	+ c(data,\delta)\int_{0}^{\tau} \int_{\Omega}
	P_{\delta}(\ro)
	\dx \dt
	+ \alpha \|\nabla(\u-\uin)\|_{L^2((0,\tau)\times\Omega)}^2
	\nonumber\\
	&\leq c(data,\alpha,\delta,\mu_0,T),
	\end{align*}
from which it follows
	\begin{multline*}
	\frac12\int_{\Omega}\ro^2(\tau)\,\dx
	+\frac12\int_{0}^{\tau}\int_{\partial\Omega}\ro^2\abs{\uB\cdot\en}\,\de{S}\dt
	+\epsil \int_{0}^{\tau}\int_{\Omega} \abs{\nabla\ro}^2\,\dx\dt
	\\
	\leq
	\alpha\int_0^\tau\int_{\partial\Omega}
	\ro^2\abs{\uB\cdot\en}\,\de{S}\dt
	+ c(data,\alpha,\delta,\mu_0,T).
	\end{multline*}
Finally, for $\alpha>0$ sufficiently small, we obtain
	\begin{equation}\label{esteps3}
	\norm{\ro}_{L^{\infty}(0,T;L^{2}(\Omega))}
	+\sqrt{\epsil}\norm{\nabla\ro}_{L^{2}(0,T;L^{2}(\Omega))}
	+\norm{\ro\abs{\uB\cdot\en}^{\frac{1}{2}}}_{L^{2}(0,T;L^{2}(\partial\Omega))}
	\leq c(data,\delta,\mu_0,T).
	\end{equation}

\begin{remark}
Let us note that in contrast to \cite{Feir03} but in correspondence to \cite{ChJiNo19}, we obtain for positive $\epsil$ estimates which depend also on $\delta.$
\end{remark}

\subsection{Solidification}
Let $h>0$ such that $\dist(S,\partial \Omega)>h$, and let us define the interior part of the solid as follows
$$
O=\{ \vecx\in S| \dist (\vecx,\partial S ) > r \},
\qquad r>0 \text{ small enough}.
$$
Next, we introduce a sequence of functions $\eta_n$ representing regularized deformations of the rigid body
given by
\begin{equation}
\dfrac{\de{}}{\dt} \eta_n[t](\vecy)=R_{r}[\u_n]\bigl(t,\eta_n[t](\vecy)\bigr),
\qquad \eta_n[0](\vecy)=\vecy,
\end{equation}
and a sequence of sets defined by
\begin{equation*}
O_n(t)=\eta_n[t](O).
\end{equation*}
Here, the regularization $R_{r}[\u_n]$ is defined using convolution
$$R_{r}[\u_n] (t,\vecx) 
= (\u_n*\omega_r)(t,\vecx) 
= \int_{\R^3}\omega_r (\vecx-\vecy) {\u_n}(t,\vecy) \de{\vecy},$$
where $\omega_r\in C^\infty(\R^3)$ is non-negative, radially symmetric and radially non-increasing function such that $\supp \omega_r \subset B_{r}(0)$ and $\int_{\R^3}\omega_r \dx =1$.

To further describe our construction, we introduce the signed distance function from the boundary of the set $S\subset\mathbb{R}^3$ as follows
$$\db_S(\vecx) = \dist\big(\vecx, \overline{\R^3\setminus S}\big)
-\dist\big(\vecx, \overline{S}\big),
$$
where for closed non-empty set $F\subset\R^3$ we denote the distance
$\dist(\vecx, F) = \min_{\vecy\in F}|\vecx-\vecy|.$
Now, we define $$\chi_n= \db_{O_n(t)}(\vecx).$$
It is possible to choose $\uin$ in such a way that it satisfies the condition:
$
\supp\uin\subset U_{\frac{1}{2}h}
= \{ \vecx\in \R^3| \dist (\vecx,\partial\Omega ) < \frac{h}{2} \},
$
and then define
\begin{equation} 	\label{artviscosity}
\mu_n (\chi_n) =\mu + n H(\chi_n+\reg)\xi \text{ and }\lambda_n(\chi_n)=\lambda +n H(\chi_n+\reg)\xi
\end{equation}
with $H$ a smooth non-negative convex function such that $H(z)=0$ for $z\leq0$ and $H(z)>0$ for $z>0,$ and smooth cut-off function $\xi\in C^{\infty}(\R^3)$ such that $\xi=0$ on $U_{\frac{1}{2}h}$ and $\xi=1$ on $\Omega\setminus U_{h}$ (see Figure \ref{plot:domain}). Note that 
$\mu_n(\chi_n)\geq\mu$, $\mu_n(\chi_n)+\lambda_n(\chi_n)\geq 0$ and the cut-off function $\xi$ is chosen to provide that $\mu_n(\chi_n)=\mu$, $\lambda_n(\chi_n)=\lambda$ on $\supp\uin$. 
Therefore, the assumptions of Lemma \ref{existapprox} are satisfied.

\begin{figure}
	\centering
	\includegraphics{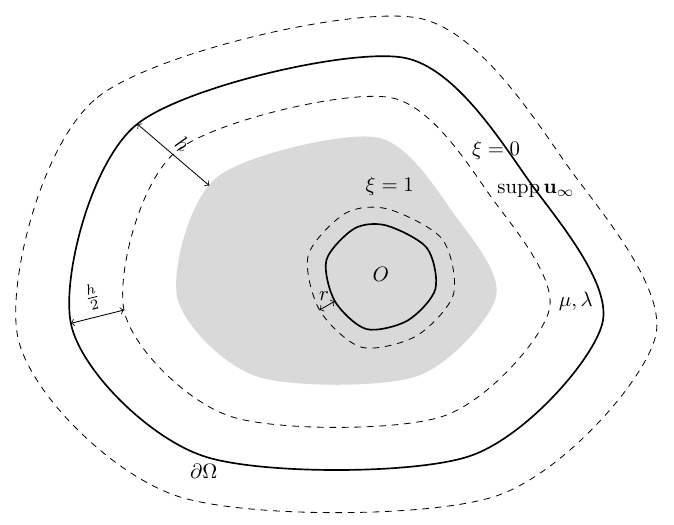}
	\caption{}
	\label{plot:domain}
\end{figure}

Now, by estimates obtained by Lemma \ref{existapprox}
we can choose a suitable convergent subsequences (denoted in the same way)
\begin{align}
	\ro_n\rightarrow \ro \quad&\text{in } L^{\beta}((0,T)\times\Omega)
	\\
	\u_n\rightarrow \u \quad&\text{in } L^{2}(0,T;W^{1,2}(\Omega))
	\\
	(\ro_n\u_n) \rightharpoonup (\ro\u)\quad &\text{in } L^{2}((0,T)\times\Omega)\\
	\nabla\ro_n\rightharpoonup \nabla\ro \quad&\text{in } L^{2}((0,T)\times\Omega)
\end{align}

Now we want to find the limit for the sequence of rigid bodies. We will use the following lemma (see Proposition 5.1. from \cite{Feir03}, see also \cite{Sche19} for a detailed proof).
\begin{lem}\label{isometries_limit}
Let $\vektor{v}_n = \vektor{v}_n(t,\vecx)$ be a family of Carath{\'e}odory functions such that
$$t\mapsto \Bigl( \norm{\vektor{v}_n(t)}_{L^\infty(\R^3)} + \norm{\nabla\vektor{v}_n(t)}_{L^\infty(\R^3)} \Bigr)\text{ is bounded in }L^2(0,T).$$ Let $\eta_n[t]:\R^3 \mapsto \R^3$ be the solution operator representing the characteristic curves generated by $\vektor{v}_n,$ id est
$$\dfrac{\de{}}{\de{t}}\eta_n[t](\vecx) = \vektor{v}_n\bigl(t,\eta_n[t](\vecx)\bigr),\quad \eta_n[0](\vecx)=\vecx,\: \vecx\in\R^3. $$
Finally, let us define for  $O\subset \R^3$, $O_n(t) = \eta_n[t](O). $
Then there exists a subsequence such that
$$\eta_n[t] \to \eta[t] \text{ in } C_{loc}(\R^3) \text{ as } n \to \infty \text{ uniformly in } t \in[0,T],$$
where
\begin{equation}
O(t) = \eta[t](O),
\qquad \frac{d}{dt} \eta[t](\vecx)=\vektor{v}\bigl(t,\eta[t](\vecx)\bigr),
\qquad \eta[0](\vecx)=\vecx,\:\vecx\in \R^3,
\end{equation}
and $\vektor{v}_n$ converges to $\vektor{v}$ weakly-* in $L^2(0,T; W^{1,\infty}(\R^3)).$ Moreover, $\db_{O_n(t)} \to \db_{O(t)}$ in $C_{loc}(\R^3)$ uniformly in $t\in[0,T].$
\end{lem}
We apply this lemma on regularized velocity fields  $\vektor{v}_n = R_{r}[\u_n]$ which clearly satisfy the hypotheses.
Thus, we obtain
\begin{align} \label{convergence_chi}
	\chi_n(\vecx) = \db_{O_n(t)}\to \db_{O(t)}\quad\text{in } C_{loc}(\R^3)\text{ uniformly in } t\in[0,T] \\
 \frac{d}{dt} \eta[t](\vecx)=R_{r}[\u]\bigl(t,\eta[t](\vecx)\bigr),
\qquad \eta[0](\vecx)=\vecx.
\end{align}

Since $\dist(O,\partial\Omega)> h+r$ and $t\mapsto \dist(O(t),\partial\Omega)$ is continuous (we control $\u$, thus also the time derivative of $\eta$), there exists $T'\in(0,T]$ such that 
$\dist(O(t),\partial\Omega)\geq h+r$ for all $t\in[0,T']$, so we can choose
\begin{equation}\label{Tmax}
T_{max}(\epsil,\delta)=\max\bigl\{ T'\in(0,T]\, |\, \dist(O(t),\partial\Omega)\geq h+r, \forall t\in[0,T'] \bigr\}>0
\end{equation}
and consider the solution on the time interval $[0,T_{max}]$. 
Now, we define open sets
\begin{align*}
	&Q^f = \{ (t,\vecx)\in (0,T_{max})\times\Omega \,|\, \db_{O(t)}(\vecx)+r <0  \},
	\\
	&Q^S = \{ (t,\vecx)\in (0,T_{max})\times\Omega \,|\, \db_{O(t)}(\vecx)+r >0  \}.
\end{align*}
Let $K^f$ be some compact subset of $Q^f$. 
Then, as in \cite{Feir03},
\begin{equation}
	\mu_n(\chi_n)=\mu,\quad \lambda_n(\chi_n)=\lambda
\end{equation}
on any compact $K^f\subset Q^f$ and for any $n\geq n(K^f)$ for some constant $n(K^f)\in\mathbb{N}$.

Let $K^S$ be some compact subset of $Q^S$. Then there exists some constant $r(K^S)>0$ such that 
\begin{equation}
	\db_{O(t)}(\vecx)+r\geq r(K^S)>0,
	\quad\forall (t,\vecx)\in K^S.
\end{equation}
By \eqref{convergence_chi} there exists $n(K^S)\in\mathbb{N}$ such that
\begin{equation}
\abs{\db_{O_n(t)}(\vecx)-\db_{O(t)}(\vecx)}
<\frac{r(K^S)}{2},
\quad\forall (t,\vecx)\in K^S, n\geq n(K^S).
\end{equation}
Therefore,
\begin{equation}
	\db_{O_n(t)}(\vecx)+r= \db_{O_n(t)}(\vecx)- \db_{O(t)}(\vecx) +\db_{O(t)}(\vecx)+r
	\geq \frac{r(K^S)}{2}>0.
\end{equation}
Function $H$ from \refx{artviscosity} is increasing and so
\begin{equation*}
H(\db_{O_n(t)}(\vecx)+r)\geq H\left(\frac{r(K^S)}{2}\right) >0,
\quad \forall  n\geq n(K^S).
\end{equation*}
For $(t,\vecx)\in K^S$, we have $\dist(\vecx,\partial\Omega)> h$, which means that $\xi(\vecx)=1$. Thus,
$$
\mu_n(\chi_n)
= \mu + n H(\db_{O_n(t)}(\vecx)+r)\xi(\vecx)
\geq \mu+n c
$$
where constant $c>0$ is independent of $t$ and $\vecx$.
This implies that
\begin{equation}\label{RigidLimit}
	\field{D}(\u_n)\to 0\quad\text{in } L^2(K^S),
\end{equation}
because otherwise
\begin{multline*}
	\int_{K^S} 2\mu_n(\chi_n)\abs{\field{D}(\u_n-\uin)}^2\,\dx\dt
	\geq 2(\mu+nc)\int_{K^S} \abs{\field{D}(\u_n-\uin)}^2\,\dx\dt\\
	= 2(\mu+nc)\int_{K^S} \abs{\field{D}(\u_n)}^2\,\dx\dt
	\to \infty
\end{multline*}
which is a contradiction to the bound obtained from the energy inequality. Note that $\uin=\boldsymbol{0}$ on $ K^S$.

 As in \cite{Feir03}, \eqref{RigidLimit} implies that $\field{D}(\u)=0$ on $Q^S$ and it can be shown that $\eta[t]$ is an isometry and $\u$ compatible with $\eta[t]$.

In contrast to the previous work of Feireisl \cite{Feir03}, we do not allow the possible contact of the rigid body with the boundary of the domain. This is actually natural, as the outflow velocity is generally not rigid, thus not compatible with the motion of the rigid body. Therefore, we want to obtain the lower estimate on the minimum existence time $T_{max}$ independent of the approximation parameters $\epsil$ and $\delta.$ 

Let us define $d=\dist(S,\partial\Omega)>h$ for $h>0$ and $t\in (0,T]$. For $\vektor{x}_0\in O$ and $\vecx=\eta[t](\vektor{x}_0)\in O(t)$ we compute
\begin{align*}
|\dist(\vecx,\partial\Omega)-\dist(\vektor{x}_0,&\:\partial\Omega)|
\leq |\vecx-\vektor{x}_0|
= |\eta[t](\vektor{x}_0)-\eta[0](\vektor{x}_0)|
\leq \left|\int_{0}^{t}\partial_t\eta[\tau](\vektor{x}_0)\,\dt\right|\\
&\leq t^{\frac{1}{2}}\|\partial_t\eta\|_{L^2(0,T;L^{\infty}(O))}
\leq t^{\frac{1}{2}}\|R_r[\u]\|_{L^2(0,T;L^{\infty}(\R^3))}
\leq t^{\frac{1}{2}}c\|\u\|_{L^2(0,T;W^{1,2}(\Omega))}\\
&\leq t^{\frac{1}{2}}c(data, T)
\end{align*}
and, since $\dist(O,\partial\Omega)=d+r>h+r$, $r>0$, we obtain
\begin{align*}
\dist(\vecx,\partial\Omega)
&\geq \dist(\vektor{x}_0,\partial\Omega) - |\dist(\vecx,\partial\Omega)-\dist(\vektor{x}_0,\partial\Omega)|
\geq d+r - t^{\frac{1}{2}}c(data, T).
\end{align*}
Now, for 
$$T_0=T_0(data, d, h, T)=\min\left\lbrace T,\left(\frac{d-h}{c(data, T)}\right)^2\right\rbrace>0,$$
we have
\begin{equation*}
\dist(\vecx,\partial\Omega)
>h+r,\quad\forall t\in(0,T_0],\,\forall \vecx\in O(t)
\qquad\Leftrightarrow\qquad \dist(O(t),\partial\Omega)
>h+r,\quad\forall t\in(0,T_0].
\end{equation*}
We conclude that
\begin{equation}
    T_{max}=T_{max}(data, d, h, \epsil,\delta, T)\geq T_0(data, d, h, T)>0. \label{TMax}
\end{equation}

For the momentum and continuity equations, we just comment that for a positive $\epsil$ we have sufficient regularity of the density and velocity fields to ensure convergence in all terms (see \cite[Section 6]{Sche19} for details), allowing us to obtain the following.

\begin{lem} \label{positiveepsil}
Let $\Omega\subset \R ^3$ and suppose that $\epsil>0$ and $T\leq T_{max}$. Suppose that there exist positive numbers $0<\underline{\ro}<\overline{\ro}$ such that the initial and boundary data satisfy
$$\u_0 \in L^2(\Omega),\quad   \ro_{0,\delta} \in W^{1,2}(\Omega),\quad\ro_B\in C(\partial\Omega),\quad \ro_{0,\delta},\ro_B\in\Bigl[\underline{\ro},\overline{\ro}\Bigr] ,  $$ then for $\uin$ an~extension of the boundary data $\uB$ from Lemma \ref{extend} there exists a triple $(\ro, \u, \eta)$ such that 
\begin{enumerate}
     \item the functions $\ro$ and $\u$ lie in regularity classes
\begin{gather*}
\ro\geq0 \text{ a.e. in }Q_T,\quad \ro\in L^{\infty}\bigl(0,T;L^{\beta}(\Omega)\bigr)\cap L^{r}\bigl(0,T;W^{1,r}(\Omega)\bigr),\:r>1  \\
\u\in 
L^2\bigl(0,T;W^{1,2}(\Omega)\bigr),\quad
\u-\uin \in L^2\bigl(0,T;W^{1,2}_0(\Omega)\bigr), \quad \ro\u \in C_{\text{weak}}([0,T];L^1(\Omega)),
\end{gather*}
 and $\u$ is {\it compatible} with parametrized isometries $\eta[t]$ describing the evolution of $S(t),$
 \item the continuity equation and its renormalized version are satisfied in the sense of \refx{CEeps} and \refx{RCEeps},
 \item the weak variational formulation of the momentum equation holds true
\begin{multline}
\int_{\Omega}{(\ro\u\cdot\fib)(\tau,.)}\dx - \int_{\Omega}{\vektor{m}_0(.)\cdot\fib(0,.)}\dx
=  - \epsil \int_{0}^\tau \int_{\Omega} \bigl(   \nabla \ro\cdot \nabla \u \cdot \fib\bigr) \dx\dt\\ + \int_{0}^\tau \int_{\Omega} \Bigl(\ro\u \cdot\partial_t \fib +(\ro\u\otimes\u) :\field{D}(\fib)  + p_\delta(\ro)\Div\fib -\field{S}(\u):\field{D}(\fib) \Bigr) \dx\dt \\
\forall\tau\in[0,T],\:\forall \fib \in \mathcal{R}(\overline{Q}^S) = \bigl\{\fib\in C^{\infty}_c(Q_T)|\field{D}(\fib)=0\text{ on some neighbourhood of }\overline{Q}^S \bigr\},
\end{multline}
\item the energy inequality \refx{EIeps} is valid with $\field{S}$ instead of $\field{S}_{\chi,n}.$
\end{enumerate}
Moreover, we have estimates \refx{esteps}, \refx{esteps2} and \refx{esteps3}.
\end{lem}

\subsection{Limit $\epsil\to 0$}
\label{Sec:ArtificialViscosity}

In this section our goal is to prove the following lemma.
\begin{lem} \label{positivedelta}
	Let $\Omega\subset\R^3$ be a bounded domain of class $C^{2+\nu}$, $\nu>0$ and let $h>0$. Suppose that there exist positive numbers $0<\underline{\ro}<\overline{\ro}$ such that the initial and boundary data satisfy
	$$
	(\ro\u)_0 \in L^{2}(\Omega),\quad   \ro_{0,\delta} \in W^{1,2}(\Omega),
	$$
	$$  
	\ro_B\in C(\partial\Omega),
	\quad\uB\in C^2(\partial\Omega),
	\quad \ro_{0,\delta},\ro_B\in\Bigl[\underline{\ro},\overline{\ro}\Bigr]
	$$ 
	and the corresponding pressure $p$ is determined by
	\begin{equation*}
	p(\ro) = a\rho^{\gamma}+ \delta\ro^{\beta}, \qquad
	\beta>\max\left\lbrace \frac{9}{2},\gamma\right\rbrace,
	\,\gamma>\frac{3}{2} .
	\end{equation*}
	Then for $\uin$ an~extension of the boundary data $\uB$ from Lemma \ref{extend} there exist a time $T=T(h)>0$ and a weak solution in the sense of Definition \ref{def_weak} such that $\dist(S(t),\partial\Omega)\geq h$, for all $t\in[0,T]$.
\end{lem}

Let $(\ro_{\epsil}, \u_{\epsil}, \eta_{\epsil})$ be a solution obtained by Lemma \ref{positiveepsil}, and let $T=T_{0}(h)>0$ arising from \refx{TMax}. We will first pass to the limit in the weak formulations of the continuity and momentum equation.
	
By Lemma \ref{positiveepsil} we have the following estimates
\begin{align}
&\norm{\ro_{\epsil}\abs{\u_{\epsil}}^2}_{L^{\infty}(0,T;L^1(\Omega))}
\leq c(data,\mu,T),
\\
&\norm{\u_{\epsil}}_{L^{2}(0,T;W^{1,2}(\Omega))}
\leq c(data,\mu,T),
\\
&\norm{\ro_{\epsil}}_{L^{\infty}(0,T;L^{\beta}(\Omega))}
\leq c(data,\delta,\mu,T),
\\
&\sqrt{\epsil}\norm{\nabla\ro_{\epsil}}_{L^{2}(0,T;L^{2}(\Omega))}
\leq c(data,\delta,\mu,T),
\end{align}
which implies	
\begin{align}
\ro_{\epsil}\stackrel{*}{\rightharpoonup} \ro \quad&\text{in } L^{\infty}(0,T;L^{\beta}(\Omega))
,\\
\u_{\epsil}\rightharpoonup \u \quad&\text{in } L^{2}(0,T;W^{1,2}(\Omega)),\\
\ro_{\epsil}\u_{\epsil}\stackrel{*}{\rightharpoonup}\overline{\ro\u} \quad&\text{in } L^{\infty}(0,T;L^{\frac{2\beta}{\beta+1}}(\Omega))
\label{weaklimit_rou},\\
\ro_{\epsil}\u_{\epsil}\otimes\u_{\epsil}\rightharpoonup \mathbb{P} \quad&\text{in } L^{1}(Q_T),
\label{weaklimit_convective}
\end{align}
and terms with $\epsil$ in equations vanish
\begin{align*}
&\epsil \int_0^\tau \int_{\Omega}\nabla\ro_{\epsil}\cdot \nabla\psi \dx\dt
\rightarrow 0,
\\
& \epsil \int_{0}^\tau \int_{\Omega} \nabla \ro_{\epsil}\cdot \nabla \u_{\epsil} \cdot \fib \dx\dt
\rightarrow 0.
\end{align*}
By the Arzel\`{a}-Ascoli theorem, from the continuity equation and the uniform bounds above, it can be shown that
\begin{equation}\label{strong_conv_ro_eps}
\varrho_{\varepsilon} \rightarrow \varrho \text { in } C_{\text {weak }}\left([0, T] ; L^\beta(\Omega)\right).
\end{equation}
From embedding $L^\beta(\Omega) \hookrightarrow \hookrightarrow W^{-1,2}(\Omega)$ it follows that
\begin{equation*}
\varrho_{\varepsilon} \rightarrow \varrho \text { in } C_{\text {weak }}\left([0, T] ; W^{-1,2}(\Omega)\right),
\end{equation*}
which together with
$\mathbf{u}_{\varepsilon} \rightharpoonup\mathbf{u}$ in $L^2\left(0, T ; W^{1,2}(\Omega)\right)$ implies
\begin{align*}
\int_{0}^{T}\int_{\Omega}\left(\ro_{\epsil}\u_{\epsil}-\ro\u\right)\dx\dt
= \int_{0}^{T}\int_{\Omega}\left(\ro_{\epsil}-\ro\right)\u_{\epsil}\dx\dt
+\int_{0}^{T}\int_{\Omega}\ro\left(\u_{\epsil}-\u\right)\dx\dt\to 0.
\end{align*}
Therefore,
\begin{align*}
&\ro_{\epsil}\u_{\epsil}\stackrel{*}{\rightharpoonup} \ro\u \quad \text{in } L^{\infty}(0,T;L^{\frac{2\beta}{\beta+1}}(\Omega))
\end{align*}
and by letting $\epsil\to 0$ in the continuity equation \eqref{RCEeps} we obtain that $(\ro,\u)$ satisfies continuity equation \eqref{continuity_weak}, as well as its renormalized version by using the following Lemma 3.1 from \cite{ChJiNo19}.

\begin{lem}
	Suppose that $\Omega\subset\R^3$ is a bounded Lipschitz domain,
	and let $\ro_B,\u_B$ satisfy assumptions of Theorem \ref{main}. Assume that the inflow portion of the
	boundary $\Gamma_{in}$ is an  open $(d-1)$-dimensional manifold of class $C^2$. Suppose further that couple
	$(\ro,\u)\in L^2(Q_T)\times L^2(0,T;W^{1,2}(\Omega))$ satisfies the continuity equation in the weak sense \eqref{continuity_weak}.
	Then $(\ro,\u)$ also satisfies renormalized continuity equation \eqref{renormalized_weak}.
\end{lem}

Compatibility of the velocities $\u_{\epsil}$ with the isometries $\{\overline{S},\eta_{\epsil}[t]\}$ gives
\begin{equation*}
	\u_{\epsil}(t,\vecx) = \u_{\epsil}^{S}(t,\vecx) 
	= \mathbf{V}_{\epsil}(t) + \field{Q}_{\epsil}(t)(\vecx-\vecX_{\epsil}(t))
\end{equation*}
for almost all $\vecx\in\overline{S}$ and almost all $t\in(0,T)$, where
\begin{equation*}
\frac{d}{dt}\vecX_{\epsil}(t) = \mathbf{V}_{\epsil}(t),
\quad
S_{\epsil}(t) = \eta_{\epsil}[t](S).
\end{equation*}
Since $\u_{\epsil}$ is bounded in $L^2(0,T;W^{1,2}(\Omega))$, functions $\mathbf{V}_{\epsil}$, $\field{Q}_{\epsil}$ are bounded in $L^2(0,T)$, and therefore $\u_{\epsil}^{S}$ is bounded in $L^2(0,T;W^{1,\infty}(\Omega))$. 
Using the extension theorem for Sobolev functions, we can now extend velocities $\u_{\epsil}^{S}$ to functions $\mathbf{v}_{\epsil}$, which are bounded in $L^2(0,T;W^{1,\infty}(\R^3))$. Next, we apply Lemma \ref{isometries_limit} with these functions $\mathbf{v}_{\epsil}$ to obtain
\begin{equation*}
\eta_{\epsil}[t] \to \eta[t] \text{ in } C_{loc}(\R^3) \text{ as } {\epsil} \to 0 \text{ uniformly in } t \in[0,T]
\end{equation*}
and
\begin{equation}\label{covergence_body}
S_{\epsil}(t) \stackrel{b}{\rightarrow} S(t), \quad \text { with } S(t)=\eta[t]\left(S\right)
\end{equation}
which further implies that $\{\overline{S},\eta[t]\}$ is compatible with the limit velocity $\u$.

In the previous section, time $T_0(h)>0$ is chosen so that $\dist(S_{\epsil}(t),\partial \Omega)\geq h$ for all $\epsil\in(0,\epsil_0)$ and for all $t\in[0,T_0(h)]$. Now the following Lemma implies that $\dist(S(t),\partial \Omega)\geq h$, for all $t\in[0,T_0(h)]$.
    \begin{lem}
		Let $h>0$ and let $S_{n}\subset\Omega$ be a sequence of sets satisfying $\dist(S_{n},\partial \Omega)\geq h$, for all $n\in\mathbb{N}$. If $S_{n} \stackrel{b}{\rightarrow} S$, then $\dist(S,\partial \Omega)\geq h$.
    \end{lem}
\begin{proof}
    Let $\sigma>0$. Then there exists $n_1\in\mathbb{N}$ such that 
    $$
    \norm{\db_{S_{n}}-\db_{S}}_{L^{\infty}(\Omega)}<\sigma,\qquad \forall n\geq n_1,
    $$
    whence, for $\vecy\in S\setminus S_{n}$, we have
    $$
    \dist\bigl(\vecy,\overline{S_{n}}\bigr) + \dist\bigl(\vecy,\overline{\R^3\setminus S}\bigr)<\sigma,
    $$
    which implies that
    $$
    S\subset S_{n}^{\sigma}
		=\left\lbrace \vecy\in \Omega\,:\, \dist(\vecy,S_{n})<\sigma\right\rbrace,\qquad \forall n\geq n_1.
    $$
    Now, we obtain that
    $$
    \dist(S,\partial \Omega)\geq \dist(S_{n}^{\sigma},\partial \Omega)
    \geq \dist(S_{n},\partial \Omega)-\sigma
    \geq h-\sigma.
    $$
    By letting $\sigma\to 0$, we conclude $\dist(S,\partial \Omega)\geq h$.
\end{proof}

We define sets
\begin{align*}
\overline{Q}_{\epsil}^S & =\left\{(t, \vecx) \in[0, T] \times \overline{\Omega} \:|\: \vecx \in  \eta_{\epsil}[t]\left(\overline{S}\right)\right\} ,\\
\overline{Q}^S & =\left\{(t, \vecx) \in[0, T] \times \overline{\Omega}  \:|\: \vecx \in \eta[t]\left(\overline{S}\right)\right\}, \\
Q_{\epsil}^f & =((0, T) \times \Omega) \backslash \overline{Q}_{\epsil}^S, \\
Q^f &=((0, T) \times \Omega) \backslash \overline{Q}^S .
\end{align*}
From \eqref{covergence_body} it can be shown that
\begin{itemize}
\item for any test function $\fib\in\mathcal{R}(\overline{Q}^S)$ there exists $\epsil(\fib)>0$ such that $\fib\in\mathcal{R}(\overline{Q}_{\epsil}^S)$ for all $0<\epsil<\epsil(\fib)$, and
\item for any compact $K^f\subset Q^f$ there exists $\epsil(K^f)>0$ such that  $K^f\subset Q_{\epsil}^f$ for all $0<\epsil<\epsil(K^f)$,
\end{itemize}
which is needed to pass to the limit in the momentum equation.

Now we can identify the limit $\mathbb{P}$ from \eqref{weaklimit_convective} on the fluid part
\begin{equation}\label{identityP}
\mathbb{P} = \varrho \mathbf{u} \otimes \mathbf{u}
\text{ a.e. on } Q^f.
\end{equation}
This is enough since $\field{D}(\fib) = 0$ on $\overline{Q}^S$ for any test function $\fib\in\mathcal{R}(\overline{Q}^S)$, and the term $\mathbb{P}:\field{D}(\fib)$ coincides with $\varrho \mathbf{u} \otimes \mathbf{u}:\field{D}(\fib)$ almost everywhere on $Q_T$. As in \cite{Feir03}, the identification \eqref{identityP} is just a consequence of embedding $L^{\frac{2\beta}{\beta+1}}(U)\hookrightarrow\hookrightarrow W^{-1,2}(U)$ and the convergence 
\begin{equation*}
\ro_{\epsil}\u_{\epsil}\to \ro\u \quad \text{in } C_{weak}(\overline{J};L^{\frac{2\beta}{\beta+1}}(U)),
\end{equation*}
where $J\times U$ is a neighbourhood of some arbitrary point $(t,\vecx)\in Q^f$ such that $\overline{J\times U}\subset Q^f$, which can be deduced from the momentum equation, see \cite{Sche19} for details.

Next, we will need some better estimate for the density in order to pass to the limit in the pressure on the fluid region, this is provided by corresponding \cite[Lemma 8.1]{Feir03}, which we restate here.
\begin{lem}\label{estimate_rho}
For any compact $K^f\subset Q^f$ we have 
\begin{equation}\label{estimate_rho_1}
\norm{\ro_{\epsil}}_{L^{\beta+1}(K^f)}
+ \norm{\ro_{\epsil}}_{L^{\gamma+1}(K^f)}
\leq c(data,\delta,\mu,T, K^f).
\end{equation}
\end{lem}
The lemma involves only the situation inside $Q^f$, the proof uses the standard Bogovskii operator and follows exactly the same lines as in \cite{Feir03}.

Lemma \ref{estimate_rho} implies that
\begin{equation*}
p_\delta(\ro_{\epsil})\rightharpoonup\overline{p_\delta(\ro)}
\quad \text{ in } L^{\frac{\beta+1}{\beta}}(K^f),
\text{ for any compact } K^f\subset Q^f.
\end{equation*}

Altogether, we can let $\epsil\to 0$ in the momentum equation, since we have shown that every term in the weak formulation converges, we obtain
\begin{multline}
\int_{\Omega}{(\ro\u\cdot\fib)(\tau)}\dx - \int_{\Omega}{\vektor{m}_0\cdot\fib(0)}\dx
\\ 
= \int_{0}^\tau \int_{\Omega} \Bigl(\ro\u \cdot\partial_t \fib +(\ro\u\otimes\u) :\field{D}(\fib)  + \overline{p_\delta(\ro)}\Div\fib -\field{S}(\u):\field{D}(\fib) \Bigr) \dx\dt.
\end{multline}
It remains to show
\begin{equation}\label{pressure_viscositylimit}
	\overline{p_\delta(\ro)} = p_\delta(\ro),
\end{equation}
id est to prove strong convergence of the density 
\begin{equation*}
\ro_{\epsil}\to \ro
\quad \text{ in } L^{1}(Q^f).
\end{equation*}

\subsubsection{Strong convergence of density}
\begin{remark}\label{renormalized_b}
By using the Lebesgue dominated convergence theorem it can be shown that the family of functions $b$ in the renormalized equation can be extended for $\ro\in L^{\infty}\bigl(0,T;L^{p}(\Omega)\bigr)$ to
\begin{equation*}
b \in C[0, \infty) \cap C^1(0, \infty), z b^{\prime}-b \in C[0, \infty),|b(z)| \leq c\left(1+z^{5 p / 6}\right),\left|z b^{\prime}(z)-b(z)\right| \leq c\left(1+z^{p / 2}\right).
\end{equation*}
\end{remark}

As in \cite{ChJiNo19}, according to Remark \ref{renormalized_b}, we can take $b(z)=z\log z$ in the renormalized version of continuity equation \eqref{renormalized_weak}
\begin{multline}\label{renormalized_eq1}
\int_{\Omega}{(b(\ro)\psi)(\tau)}\dx - \int_{\Omega}{b(\ro_{0,\delta})\psi(0)}\dx \\
= \int_{0}^\tau \int_{\Omega} \Bigl(b(\ro)\partial_t \psi +b(\ro)\u \cdot\nabla\psi -  \psi\ro\Div\u  \Bigr) \dx\dt - \int_0^\tau\int_{\Gamma_{in}}(\ro_B\log\ro_B)\uB\cdot\en \psi \de{S}\dt
\end{multline}
and in renormalized continuity equation of the approximate problem \eqref{RCEeps}
\begin{multline}\label{renormalized_eq2}
\int_{\Omega}{(b(\ro_{\epsil})\psi)(\tau)}\dx 
- \int_{\Omega}{b(\ro_{0,\delta})\psi(0)}\dx 
+ \int_0^\tau\int_{\Gamma_{in}}\Bigl(b(\ro_{\epsil}) +b'(\ro_{\epsil})(\ro_B-\ro_{\epsil})\Bigr)\uB\cdot\en \psi \de{S}\dt \\
= \int_{0}^\tau \int_{\Omega} \Bigl(b(\ro_{\epsil})\partial_t \psi 
+\bigl(b(\ro_{\epsil}) \u_{\epsil} -\epsil b'(\ro_{\epsil}) \nabla \ro_{\epsil} \bigr) \cdot\nabla\psi -  \psi\ro_{\epsil}\Div\u_{\epsil}   - \psi \epsil b''(\ro_{\epsil})|\nabla\ro_{\epsil}|^2 \Bigr) \dx\dt
\\
\leq \int_{0}^\tau \int_{\Omega} \Bigl(b(\ro_{\epsil})\partial_t \psi 
+b(\ro_{\epsil}) \u_{\epsil} \cdot\nabla\psi -  \psi\ro_{\epsil}\Div\u_{\epsil} \Bigr) \dx\dt
-
\epsil \int_{0}^\tau \int_{\Omega}(\log\ro_{\epsil}+1) \nabla \ro_{\epsil} \cdot\nabla\psi\dx\dt,
\end{multline}
for any $\tau\in[0,T]$ and any $\psi \in C_c^1\left([0,T] \times\left(\Omega \cup \Gamma_{in}\right)\right), \psi \geq 0$. Here we used the boundary conditions and definition of function $b$; the inequality follows from convexity of $b$.

Now we subtract \eqref{renormalized_eq1} from \eqref{renormalized_eq2} and take $\psi$ independent of $t\in[0,T]$
\begin{multline*}
\int_{\Omega}{b(\ro_{\epsil}(\tau))\psi}\dx 
- \int_{\Omega}{b(\ro(\tau))\psi}\dx 
+ \int_0^\tau\int_{\Gamma_{in}}\Bigl(b(\ro_B) -b'(\ro_{\epsil})(\ro_B-\ro_{\epsil})- b(\ro_{\epsil})\Bigr)|\uB\cdot\en| \psi \de{S}\dt
\\
-\int_{0}^\tau \int_{\Omega} \Bigl(\left(b(\ro_{\epsil})-b(\ro)\right) \u_{\epsil} \cdot\nabla\psi
+b(\ro)\left(\u_{\epsil}-\u\right)  \cdot\nabla\psi \Bigr) \dx\dt
+ \int_{0}^\tau \int_{\Omega} \psi\left(\ro_{\epsil}\Div\u_{\epsil}-\ro\Div\u\right)\dx\dt
\\
\leq
-\epsil \int_{0}^\tau \int_{\Omega}(\log\ro_{\epsil}+1) \nabla \ro_{\epsil} \cdot\nabla\psi\dx\dt,
\end{multline*}
convexity of the function $b$ gives
\begin{multline}\label{renormalized_eq3}
\int_{\Omega}{\ro_{\epsil}(\tau)\log \ro_{\epsil}(\tau)\psi}\dx 
- \int_{\Omega}{\ro(\tau)\log\ro(\tau)\psi}\dx 
\\
-\int_{0}^\tau \int_{\Omega} \left(\ro_{\epsil}\log\ro_{\epsil}-\ro\log\ro\right) \u_{\epsil} \cdot\nabla\psi \dx\dt
-\int_{0}^\tau \int_{\Omega}
\ro\log\ro\left(\u_{\epsil}-\u\right)  \cdot\nabla\psi \dx\dt
\\
+ \int_{0}^\tau \int_{\Omega} \psi\left(\ro_{\epsil}\Div\u_{\epsil}-\ro\Div\u\right)\dx\dt
+\epsil \int_{0}^\tau \int_{\Omega}(\log\ro_{\epsil}+1) \nabla \ro_{\epsil} \cdot\nabla\psi\dx\dt
\leq 0,
\end{multline}
for any $\tau\in[0,T]$ and any $\psi \in C_c^1\left(\Omega \cup \Gamma_{in}\right), \psi \geq 0$.

\subsubsection{Effective viscous flux}

One of the key ingredients of the proof of density convergence is the weak compactness of special quantity, called the effective viscous flux. Since the seminal work \cite{Lions93} it becomes nowadays part of the folklore to deduce by testing momentum equation with suitable test function the following lemma, see \cite[Lemma 8.2]{Feir03}.
\begin{lem}\label{effectiveviscousflux}
	For $\beta>7$ the effective viscous flux identity holds
	\begin{equation*}
	\lim_{\epsil\to 0}\int_{0}^{T}\int_{\Omega}\phi \left(p_{\delta}(\ro_{\epsil})-(\lambda+2\mu)\Div\u_{\epsil}\right)\ro_{\epsil}\,\dx\dt 
	= \int_{0}^{T}\int_{\Omega} \phi \left(\overline{p_{\delta}(\ro)}-(\lambda+2\mu)\Div\u\right)\ro\,\dx\dt
	\end{equation*} 
	for all $\phi\in\mathcal{D}(Q^f)$.
\end{lem}
Since the pressure $p_{\delta}$ is nondecreasing, Lemma \ref{effectiveviscousflux} implies
\begin{equation*}
\lim_{\epsil\to 0}\int_{0}^{T}\int_{\Omega}\phi \left(\ro\Div\u-\ro_{\epsil}\Div\u_{\epsil}\right)\,\dx\dt 
\geq 0
\end{equation*} 
and therefore
\begin{equation*}
    \ro_{\epsil}\Div\u_{\epsil}\rightharpoonup\overline{\ro\Div\u}
    \text{ in } L^1(Q_T),
\end{equation*}
where
\begin{equation*}
\overline{\ro\Div\u} \geq {\ro\Div\u}
\quad
\text{a.e. in } Q^f.
\end{equation*}
As for any compact $K^S\subset Q^S$ there is $\epsil(K^S)>0$ such that 
$\Div\u_{\epsil} = 0$ a.e. on $K^S$, for all $0<\epsil<\epsil(K^S)$, we conclude that
\begin{equation}
    \overline{\ro\Div\u}\geq \ro\Div\u\quad
    \text{a.e. in } Q_T.
\end{equation}
Now, letting $\epsil\to 0$ in \eqref{renormalized_eq3} implies
\begin{equation*}
\int_{\Omega}\left(\overline{\ro\log\ro}-\ro\log\ro
\right)(\tau)\psi\dx 
-\int_{0}^\tau \int_{\Omega} \left(\overline{\ro\log\ro}-\ro\log\ro\right) \u_{\epsil} \cdot\nabla\psi \dx\dt
\leq 0,
\end{equation*}
for any $\tau\in[0,T]$ and any $\psi \in C_c^1\left(\Omega \cup \Gamma_{in}\right), \psi \geq 0$.

Then, as in \cite{ChJiNo19}, by choosing test functions
\begin{equation*}
\psi_{\epsil}(\vecx)=\begin{cases}
1 &\text { if } \operatorname{dist}\left(\vecx, \Gamma_{out}\right)>\varepsilon \\
\frac{1}{\epsil} \operatorname{dist}\left(\vecx,\Gamma_{out}\right) &\text { if } \operatorname{dist}\left(\vecx,\Gamma_{out}\right) \leq \epsil
\end{cases}
\end{equation*}
and letting $\epsil\to 0$, we obtain 
\begin{equation*}
\int_{\Omega}\left(\overline{\ro\log\ro}-\ro\log\ro
\right)(\tau)\dx 
\leq 0,
\end{equation*}
which, from the weak lower semicontinuity of convex function $b(z)=z\log z$, yields
\begin{equation*}
\overline{\ro\log\ro}=\ro\log\ro \quad\text{a.e. in } Q_T.
\end{equation*}
As usually, see e.g. \cite[Lemma 3, Section 6.3.1]{FeKaPo16}, we can conclude
\begin{equation*}
\ro_{\epsil} \rightarrow \ro \quad \text { in } L^1(Q_T).
\end{equation*}
Since strong $L^1$ convergence and boundedness in $L^{\beta+1}$ imply strong $L^{p}$ convergence for any $p<\beta+1$, we can infer from Lemma \ref{estimate_rho}
\begin{equation*}
p_{\delta}(\ro_{\epsil}) \rightarrow p_{\delta}(\ro) \quad \text { in } L^1(K^f)
\end{equation*}
for any compact $K^f\subset Q^f.$ Note that we can also pass to the limit in renormalized continuity equation. 

\subsubsection{Energy inequality}
In order to deduce the limit energy inequality, first, we integrate inequality \eqref{EIeps} over $\tau$ from $0<\tau_1 < \tau_2 < T$
\begin{multline*}
\int_{\tau_1}^{\tau_2}\int_\Omega \Bigl(\frac12\ro_{\epsil}|\u_{\epsil}-\uin|^2 + P_\delta(\ro_{\epsil}) \Bigr)(\tau)\dx  
+ \int_{\tau_1}^{\tau_2}\int_0^\tau\int_\Omega \bigl(\field{S}(\u_{\epsil}-\uin) \bigr): \field D (\u_{\epsil} - \uin)\dx \dt \dtau
\\
\leq \int_{\tau_1}^{\tau_2}\int_\Omega \Bigl(\frac12\ro_{0,\delta}|\u_0-\uin|^2 + P_\delta(\ro_{0,\delta}) \Bigr)\dx \dtau
-\int_{\tau_1}^{\tau_2}\int_0^\tau\int_\Omega \ro_{\epsil}\u_{\epsil}\cdot\nabla\uin \cdot (\u_{\epsil}-\uin)  \dx \dt\dtau\\
- \int_{\tau_1}^{\tau_2}\int_0^\tau \int_\Omega  p_\delta(\ro_{\epsil})\Div\uin \dx \dt \dtau 
-\int_{\tau_1}^{\tau_2}\int_0^\tau\int_{\Gamma_{in}} P_\delta(\ro_B)\uB\cdot\en  \,\de{S} \dt\dtau\\
- \int_{\tau_1}^{\tau_2}\int_0^\tau\int_\Omega  \field{S}(\uin): \field D (\u_{\epsil} - \uin) \dx \dt\dtau
+ \epsil \int_{\tau_1}^{\tau_2}\int_0^\tau \int_\Omega \nabla \ro_{\epsil} \cdot\nabla (\u_{\epsil}-\uin)\cdot\uin \,\dx\dt\dtau,
\end{multline*}
where we have already omitted the non-negative terms; note that $P_{\delta}$ is convex and non-negative and $\uB\cdot\en>0$ on $\Gamma_{out}$. We can pass with $\epsil\to 0$
\begin{multline*}
\int_{\tau_1}^{\tau_2}\int_\Omega \Bigl(\frac12\ro|\u-\uin|^2 + P_\delta(\ro) \Bigr)(\tau)\dx  
+ \int_{\tau_1}^{\tau_2}\int_0^\tau\int_\Omega \bigl(\field{S}(\u-\uin) \bigr): \field D (\u - \uin)\dx \dt \dtau
\\
\leq \int_{\tau_1}^{\tau_2}\int_\Omega \Bigl(\frac12\ro_{0,\delta}|\u_0-\uin|^2 + P_\delta(\ro_{0,\delta}) \Bigr)\dx \dtau
-\int_{\tau_1}^{\tau_2}\int_0^\tau\int_\Omega \ro\u\cdot\nabla\uin \cdot (\u-\uin)  \dx \dt\dtau\\
-\int_{\tau_1}^{\tau_2}\int_0^\tau\int_{\Gamma_{in}} P_\delta(\ro_B)\uB\cdot\en  \,\de{S} \dt\dtau
- \int_{\tau_1}^{\tau_2}\int_0^\tau\int_\Omega  \field{S}(\uin): \field D (\u - \uin) \dx \dt \dtau \\
- \liminf_{\epsil \rightarrow 0}\int_{\tau_1}^{\tau_2}\int_0^\tau \int_\Omega p_\delta(\ro_{\epsil})\Div\uin \dx \dt\dtau.
\end{multline*}
Since
\begin{equation*}
\int_{\tau_1}^{\tau_2}\int_0^\tau \int_{U_h} p_\delta(\ro_{\epsil})\Div\uin \dx \dt \dtau
\geq 0,
\end{equation*}
for $h>0$ sufficiently small, and
\begin{equation*}
\liminf_{\epsil \rightarrow 0}
\int_{\tau_1}^{\tau_2}\int_0^\tau \int_{\Omega\setminus U_h} p_\delta(\ro_{\epsil})\Div\uin \dx \dt\dtau
= \int_{\tau_1}^{\tau_2}\int_0^\tau \int_{\Omega\setminus U_h} p_\delta(\ro)\Div\uin \dx \dt \dtau,
\end{equation*}
we obtain
\begin{equation*}
\liminf_{\epsil \rightarrow 0}\int_{\tau_1}^{\tau_2}\int_0^\tau \int_{\Omega} p_\delta(\ro_{\epsil})\Div\uin \dx \dt \dtau
\geq 
\int_{\tau_1}^{\tau_2}\int_0^\tau \int_{\Omega\setminus U_h} p_\delta(\ro)\Div\uin \dx \dt\dtau,
\end{equation*}
which, by letting $h\to 0$, gives
\begin{equation*}
\liminf_{\epsil \rightarrow 0}\int_{\tau_1}^{\tau_2}\int_0^\tau \int_{\Omega} p_\delta(\ro_{\epsil})\Div\uin \dx \dt \dtau
\geq 
\int_{\tau_1}^{\tau_2}\int_0^\tau \int_{\Omega} p_\delta(\ro)\Div\uin \dx \dt  \dtau.
\end{equation*}
Hence, by letting $\tau_1\to\tau_2$ and setting $\tau_2=\tau$, we obtain
\begin{multline}\label{EnIn}
\int_\Omega \Bigl(\frac12\ro|\u-\uin|^2 + P_\delta(\ro) \Bigr)(\tau)\dx  
+ \int_0^\tau\int_\Omega \bigl(\field{S}(\u-\uin) \bigr): \field D (\u - \uin)\dx \dt 
\\
\leq \int_\Omega \Bigl(\frac12\ro_{0,\delta}|\u_0-\uin|^2 + P_\delta(\ro_{0,\delta}) \Bigr)\dx -\int_0^\tau\int_\Omega \ro\u\cdot\nabla\uin \cdot (\u-\uin)  \dx \dt\\
- \int_0^\tau \int_\Omega p_\delta(\ro)\Div\uin \dx \dt -\int_0^\tau\int_{\Gamma_{in}} P_\delta(\ro_B)\uB\cdot\en  \,\de{S} \dt
- \int_0^\tau\int_\Omega  \field{S}(\uin): \field D (\u - \uin) \dx \dt.
\end{multline}

\subsection{Limit $\delta\to 0$}
In this section, we suppose that $(\ro_{\delta}, \u_{\delta}, \eta_{\delta})$ is a solution obtained by Lemma \ref{positivedelta}, and $T=T_{0}(h)>0$ is defined by \refx{TMax}.

\begin{lem}\label{deltaestimates}
The solutions constructed in the previous section possess the following uniform estimates
\begin{align}
&\norm{\ro_{\delta}\abs{\u_{\delta}}^2}_{L^{\infty}(0,T;L^1(\Omega))}
\leq c(data),
\\
&\norm{\u_{\delta}}_{L^{2}(0,T;W^{1,2}(\Omega))}
\leq c(data),
\\
&\norm{\ro_{\delta}}_{L^{\infty}(0,T;L^{\gamma}(\Omega))}
\leq c(data),
\\
&\delta^{\frac{1}{\beta}}\norm{\ro_{\delta}}_{L^{\infty}(0,T;L^{\beta}(\Omega))}
\leq c(data),
\end{align}
where "data" stands for
\begin{equation*}
E_0=\int_{\Omega} \left(\frac{1}{2}\frac{\abs{\mathbf{m}_0}^2}{\ro_0}+ P(\ro_0)\right)\dx,
\, \norm{\uin}_{W^{1,\infty}(\Omega)},\overline{\ro}_B,\underline{\ro_B}.
\end{equation*}
\end{lem}
\begin{proof}
Continuity equation yields an $L^{\infty}(0,T;L^{1}(\Omega))$ bound for the sequence $\ro_{\delta}$, and then uniform estimates follow directly from energy inequality \refx{EnIn}.
\end{proof}
Similarly as before, we obtain
\begin{align}
\ro_{\delta}\to\ro \quad&\text{in } C_{weak}([0,T];L^{\gamma}(\Omega))
\\
\u_{\delta}\rightharpoonup \u \quad&\text{in } L^{2}(0,T;W^{1,2}(\Omega))\\
\ro_{\delta}\u_{\delta}\stackrel{*}{\rightharpoonup}{\ro\u} \quad&\text{in } L^{\infty}(0,T;L^{\frac{2\gamma}{\gamma+1}}(\Omega))
\end{align}

\begin{equation*}
\eta_{\delta}[t] \to \eta[t] \text{ in } C_{loc}(\R^3) \text{ as } {\delta} \to 0 \text{ uniformly in } t \in[0,T]
\end{equation*}
and
\begin{equation}
S_{\delta}(t) \stackrel{b}{\rightarrow} S(t), \quad \text { with } S(t)=\eta[t]\left(S\right)
\end{equation}
which implies that $\{\overline{S},\eta[t]\}$ is compatible with the limit velocity $\u$.\\
We set
\begin{align*}
\overline{Q}_{\delta}^S & =\left\{(t, \vecx) \in[0, T] \times \overline{\Omega} \:|\: \vecx \in  \eta_{\delta}[t]\left(\overline{S}\right)\right\} \\
\overline{Q}^S & =\left\{(t, \vecx) \in[0, T] \times \overline{\Omega} \:|\: \vecx \in \eta[t]\left(\overline{S}\right)\right\} \\
Q_{\delta}^f & =((0, T) \times \Omega) \backslash \overline{Q}_{\delta}^S, \\
Q^f & =((0, T) \times \Omega) \backslash \overline{Q}^S ,
\end{align*}
\begin{equation*}
\ro_{\delta}\u_{\delta}\otimes\u_{\delta}\rightharpoonup \mathbb{P} \quad\text{in } L^{1}(Q_T),
\end{equation*}
where
\begin{equation*}
\mathbb{P}=\ro\u\otimes\u \quad\text{ a.e. on } Q^f.
\end{equation*}
As in \cite{Feir03}
\begin{equation}
	\norm{\ro_{\delta}}_{L^{\gamma+\theta}(K^f)}^{\gamma+\theta}
	+ \delta\norm{\ro_{\delta}}_{L^{\beta+\theta}(K^f)}^{\beta+\theta}
	\leq c(data, K^f) \quad\text{with any compact}\, K^f\subset Q^f,
\end{equation}
where $\theta>0$ is a positive constant independent of $\delta$. Therefore, we can choose subsequence such that
\begin{equation}
p_{\delta}(\ro_{\delta})\rightharpoonup
\overline{p(\ro)} = a\overline{\ro^{\gamma}}
\quad\text{in } L^{\frac{\beta+\theta}{\beta}}
\text{for any compact}\, K^f\subset Q^f.
\end{equation}

Now, we let $\delta\to 0$ in the continuity equation
\begin{multline*}
\int_{\Omega}{(\ro\psi)(\tau)}\dx - \int_{\Omega}{\ro_0\psi(0)}\dx = \int_{0}^\tau \int_{\Omega} \bigl(\ro\partial_t \psi +\ro\u \cdot\nabla\psi  \bigr) \dx\dt - \int_0^\tau\int_{\Gamma_{in}}\ro_B\uB\cdot\en \psi \de{S}\dt\\
\forall\tau\in[0,T],\:\forall \psi \in C^1_{c}\bigl([0,T]\times(\Omega\cup\Gamma_{in})\bigr)
\end{multline*}
and the momentum equation
\begin{multline}
\int_{\Omega}{(\ro\u\cdot\fib)(\tau)}\dx - \int_{\Omega}{\vektor{m}_0\cdot\fib(0)}\dx
\\ 
= \int_{0}^\tau \int_{\Omega} \Bigl(\ro\u \cdot\partial_t \fib +(\ro\u\otimes\u) :\field{D}(\fib)  + \overline{p(\ro)}\Div\fib -\field{S}(\u):\field{D}(\fib) \Bigr) \dx\dt
\\
\forall\tau\in[0,T],\:\forall \fib \in \mathcal{R}(\overline{Q}^S).
\end{multline}
Passing to the limit in renormalized continuity equation we infer that
\begin{multline} 
\int_{\Omega}{(\overline{b(\ro)}\psi)(\tau,.)}\dx - \int_{\Omega}{b(\ro_0)(.)\psi(0,.)}\dx \\
 = \int_{0}^\tau \int_{\Omega} \Bigl(\overline{b(\ro)}\partial_t \psi +b(\ro)\u \cdot\nabla\psi -  \psi\overline{(b'(\ro)\ro-b(\ro))\Div\u } \Bigr) \dx\dt - \int_0^\tau\int_{\Gamma_{in}}b(\ro_B)\uB\cdot\en \psi \de{S}\dt\\
\forall\tau\in[0,T],\:\forall \psi \in C^1_{c}\bigl([0,T]\times(\Omega\cup\Gamma_{in})\bigr), \forall b\in C^1\bigl([0,\infty)\bigr) \text{ with } b'\in C^1_c\bigl([0,\infty)\bigr).
\end{multline}

It remains to show that
\begin{equation*}
	\overline{p(\ro)}=p(\ro), \text{ and }\:\overline{b(\ro)}=b(\ro)
\end{equation*}
which is equivalent to showing that
\begin{equation*}
\ro_{\delta}\to \ro
\quad \text{ in } L^{1}(Q^f).
\end{equation*}

\subsubsection{Strong convergence of density}
Defining an auxiliary function concave on $[0,\infty)$
\begin{equation*}
T(z)= \begin{cases}  z  & \text{ on } [0,1],\\
\in [1,2] & \text{ on } (1, 3), \\
T(z)= 2 & \text{ on } [3,\infty),
\end{cases}
\end{equation*}
$$
T_k(z)= kT\left(\frac{z}{k}\right),\,k\in\mathbb{N},
$$
and repeating the procedure from the limit $\epsil\to 0$ replacing $\ro$ with $T_k(\ro)$ we deduce the following form of the effective viscous flux identity
\begin{multline}
	\lim_{\delta\to 0}\int_{0}^{T}\int_{\Omega}\phi \left(p(\ro_{\delta})-(\lambda+2\mu)\Div\u_{\delta}\right)T_k(\ro_{\delta})\,\dx\dt \\
	= \int_{0}^{T}\int_{\Omega} \phi \left(\overline{p(\ro)}-(\lambda+2\mu)\Div\u\right) \overline{T_k(\ro)}\,\dx\dt
\end{multline} 
for all $\phi\in C_c^{\infty}(Q^f)$. Using the monotonicity of the pressure, this leads to
$$
\overline{T_k(\ro)\Div\u} \geq \overline{T_k(\ro)}\Div\u\quad
\text{a.e. on } Q_T.
$$

\bigskip
In order to conclude the proof we introduce the oscillations defect measure
\begin{equation*}
\operatorname{osc}_{p}[\ro_{\delta}\rightharpoonup\ro](O)
= \sup_{k\geq 1}\left(\limsup_{\delta\to 0}\int_{O}\abs{T_k(\ro_{\delta}) - T_k(\ro)}^p\,\dx\dt\right). 
\end{equation*}
As actually $$\operatorname{osc}_{\gamma+1}[\ro_{\delta}\rightharpoonup\ro](K^S)=0,\text{ for any compact }K^S\subset Q^S$$ we can follow the strategy from \cite[Lemma 6.3.]{ChJiNo19} (see also the proof of \cite[Theorem 9.1]{Feir03}) in order to deduce 
\begin{equation}
	\operatorname{osc}_{\gamma+1}[\ro_{\delta}\rightharpoonup\ro](Q_T) < \infty.
\end{equation}	
Further, we can define
$$
L_k(z)= z\int_{1}^{z} \dfrac{T_k(s)}{s^2} \ds,\text{ id est }
zL_k'(z) - L_k(z)=T_k(z)
$$
and take $b(\ro)= L_k(\ro)$ in the renormalized continuity equation (instead of $\ro\log \ro$, which was only possible on the level $\epsil\to 0$).
Except the technical truncation, the proof of the strong convergence of density follows the same lines as before, see also \cite[Section 6.5.]{ChJiNo19}.

\begin{center}
\Large\textbf{Acknowledgements} \\[4mm]
\end{center}

\textit{This work has been supported by the Czech Science Foundation (GA\v CR) through projects GA22-01591S, (for \v S.N. and A.R.) Moreover, \it \v S. N., and A.R.   have been supported by  Praemium Academiæ of \v S. Ne\v casov\' a. Additionally, A.R. has been supported by Croatian Science Foundation under the project IP-2018-01-3706 and IP-2019-04-1140. Finally, the Institute of Mathematics, CAS is supported by RVO:67985840.}

\bibliography{knihovna}

\end{document}